\documentclass[table]{amsart}
\usepackage[T1]{fontenc}
\usepackage[utf8]{inputenc}

\usepackage{amsmath}
\usepackage{amssymb}
\usepackage{amsthm}
\usepackage{hyperref}
\usepackage{cite}
\usepackage{tkz-graph}
\usepackage[all]{xy}
\usepackage{placeins}
\usepackage{diagrams}
\usepackage{booktabs}
\definecolor{lightgray}{gray}{0.9}

\theoremstyle{definition}
\newtheorem{definition}{Definition}[section]

\newtheorem{remark}[definition]{Remark}

\theoremstyle{plain}
\newtheorem{theorem}[definition]{Theorem}
\newtheorem{lemma}[definition]{Lemma}
\newtheorem{proposition}[definition]{Proposition}
\newtheorem{corollary}[definition]{Corollary}

\DeclareMathOperator{\rk}{rk}

\DeclareMathOperator{\id}{id}
\DeclareMathOperator{\lcm}{lcm}

\DeclareMathOperator{\Hom}{Hom}
\DeclareMathOperator{\NS}{NS}
\DeclareMathOperator{\Aut}{Aut}
\DeclareMathOperator{\cha}{char}
\DeclareMathOperator{\Fr}{Fr}
\DeclareMathOperator{\End}{End}
\DeclareMathOperator{\Id}{Id}
\DeclareMathOperator{\res}{res}

\newcommand{\QQ}{\mathbb{Q}}
\newcommand{\FF}{\mathbb{F}}
\newcommand{\RR}{\mathbb{R}}
\newcommand{\CC}{\mathbb{C}}
\newcommand{\ZZ}{\mathbb{Z}}
\renewcommand{\O}{\mathcal{O}}

\title[Minimal entropy]{Automorphisms of minimal entropy on supersingular K3 surfaces}
\author{Simon Brandhorst, Víctor González-Alonso}
\address{Insitut für Algebraische Geometrie, Leibniz Universität Hannover,
	Welfengarten 1, 30167 Hannover, Germany}
\email{brandhorst@math.uni-hannover.de, gonzalez@math.uni-hannover.de}
\date{September, 2016}

\begin{document}
\begin{abstract}
 In this article we give a strategy to decide whether the logarithm of a given Salem number is realized as entropy of an automorphism of a supersingular K3 surface in positive characteristic. As test case it is proved that $\log \lambda_d$, where $\lambda_d$ is the minimal Salem number of degree $d$, is realized in characteristic $5$ if and only if $d\leq 22$ is even and $d\neq 18$.
 In the complex projective setting we settle the case of entropy $\log \lambda_{12}$, left open by McMullen, by giving the construction. A necessary and sufficient test is developed to decide whether a given isometry of a hyperbolic lattice, with spectral radius bigger than one, is positive, i.e. preserves a chamber of the positive cone.
\end{abstract}
\maketitle
\tableofcontents
\section{Introduction}

A Salem number is a real algebraic integer $\lambda>1$ which is conjugate to $1/\lambda$ and whose other conjugates lie on the unit circle. 
In each even degree $d$ there is a unique smallest Salem number $\lambda_d$. Conjecturally the smallest Salem number is $\lambda_{10}$, found by Lehmer in 1933 \cite{lehmer:lambda10}.

If $F: X \rightarrow X$ is a biholomorphic map of a compact K\"ahler surface, then its topological entropy $h(F)$ is a measure for the disorder created by subsequent iterations of $F$. In general $h(F)$ is either zero or the logarithm of a Salem number $\lambda$, which is precisely the spectral radius of the linear action $F^*$ in $H^2(X,\mathbb{Z})$.

In \cite{esnault-srinivas:algebraic_entropy} Esnault and Srinivas show that if $F:X \rightarrow X$ is an automorphism of a \emph{projective} surface $X$ over a field $\kappa$, then the order of $f=F^*$ on $\NS(X)^\perp \subseteq H^2_{et}(X,\mathbb{Q}_l)$, $l\neq \cha \kappa$, is finite. 
Hence the spectral radius of $f$ is realized already in the Néron-Severi group $\NS(X)$, and by standard arguments for isometries of hyperbolic lattices it is then a Salem number.
We can define the (algebraic) entropy $h(F)$ as the logarithm of the spectral radius of $f|\NS(X)$ and the Salem degree of $f$ as the degree of this Salem number. For complex surfaces, the standard comparison results between singular and \'etale cohomology imply that the algebraic entropy coincides with the topological one.

For projective surfaces, the Salem degree is thus bounded by the rank $\rho$ of the Néron-Severi group. For $K3$ surfaces in characteristic zero this is at most $h^{1,1}(X)=20$, due to Hodge theory. However in positive characteristic, $\rho=\rk \NS(X)=22$ is possible (these are precisely the supersingular $K3$ surfaces). Since the algebraic entropy is stable under specialization, an automorphism of Salem degree $22$ in positive characteristic cannot lift to characteristic zero and neither does any of its powers (see \cite{esnault-oguiso:non-liftability}). Proofs of existence and explicit examples of such automorphisms (for Artin invariant one) have been recently studied in a number of articles \cite{brandhorst:salem22,schuett:salem22,esnault-oguiso-yu:salem22,shimada:salem22}.
Recent preprints \cite{brandhorst:higher_artin,yu:salem22} prove that {\em every} supersingular K3 surface in odd characteristic admits an automorphism of Salem degree 22. 

The study of the entropy of $F:X \rightarrow X$ becomes trivial if $X$ has positive Kodaira dimension (e.g., if $X$ is of general type, then a power of $F$ is the identity and hence $h(F)=0$). Indeed, if $F$ has positive entropy, then $X$ is either a blow up of $\mathbb{P}^2$ in at least $10$ points, a torus, a K3-surface or an Enriques surface \cite{cantat:classification,nagata:classification}. 

Instead of considering only the Salem degree of an automorphism, in this work we focus on the existence of automorphisms of (supersingular) K3 surfaces with a given entropy, and more precisely, logarithms of the minimal Salem numbers $\lambda_d$. For complex projective K3 surfaces, it is proved in \cite{mcmullen:minimum} that $\lambda_d$ is the spectral radius of an automorphism for $d=2,4,6,8,10$ or $18$, but not if $d=14,16$ or $d\geq 20$, while the case $d=12$ is left open. As a byproduct of our work, we are able to realize also $\lambda_{12}$ in the complex case (see Appendix \ref{app:lambda12_complex}), hence proving the following

\begin{theorem}[Improvement of Theorem 1.2 in \cite{mcmullen:minimum}]\label{thm:complex}
 The value $\log \lambda_d$ arises as the entropy of an automorphism of a complex projective K3 surface if and only if $d=2,4,6,8,10,12$ or $18$. 
\end{theorem}
The proof involves methods from integer linear programming, lattice theory, number fields, reflection groups and the Torelli theorem for complex K3 surfaces. 

The main purpose of this work is to extend the tools developed for the proof of this theorem in \cite{mcmullen:siegel_disk,mcmullen:entropy_and_glue,mcmullen:minimum} to supersingular K3 surfaces in positive characteristic. 
The reason to consider the supersingular case is that there is a Torelli theorem readily available while in the non-supersingular case most automorphisms (all for $p\geq 23$) lift to characteristic zero (cf \cite{jang:lift}) and can be treated there. 
In order to illustrate the techniques, we prove the following
\begin{theorem}
 The value $\log \lambda_d$ arises as the entropy of an automorphism of a supersingular K3 surface over a field of characteristic $p=5$ if and only if $d\leq 22$ is even and $d\neq 18$.
\end{theorem}
Here $p=5$ is chosen because it is the smallest prime for which the crystalline Torelli theorem is fully proven. The same methods apply for any other $p\geq 5$.  
They handle a single Salem number and one characteristic at a time (sometimes we can deal with $p$ ranging in an arithmetic progression in the spirit of \cite{brandhorst:salem22,schuett:salem22}).

In what follows we highlight some of the differences and challenges between the complex and the supersingular cases. 
Let $\lambda$ be a Salem number, $s(x)$ its minimal polynomial.

In the complex case let $F: X \rightarrow X$ be an automorphism of a projective K3 surface over $\CC$ with $h(F)=\log \lambda$. 
The singular cohomology $H^2(X,\ZZ)$ carries an integral bilinear form turning it into an even $\emph{unimodular}$ lattice of signature $(3,19)$, on which $f=F^*$ acts as an isometry. It respects further structure such as the Hodge decomposition and the ample cone in
$\NS(X) \otimes \mathbb{R} \subseteq H^2(X,\RR)$. The Torelli theorem states that this datum determines the pair $(X,f)$ up to isomorphism and conversely, that each (good) datum is coming from such a pair. 
So, in order to construct examples one has to produce a certain lattice together with a (suitable) isometry on it. 

The characteristic polynomial of $f$ factors as
\[\chi(f|H^2(X,\ZZ))=s(x) c(x)\]
where $c(x)$ is a product of cyclotomic polynomials.
The Salem and cyclotomic factors are defined then as
\[S:=\ker s(f|H^2(X,\ZZ)) \quad C:=\ker c(f|H^2(X,\ZZ)).\]
They are lattices of signatures $(1,d-1)$ and $(2,20-d)$,
$C=S^\perp$ and $S\oplus C$ is of finite index in $H^2(X,\ZZ)$.
From the unimodularity of the latter we get an isomorphism (called glue map) of discriminant groups
\[A_S\cong A_C\] compatible with the action of $f$.
It follows that the polynomials $s(x)$ and $c(x)$ have a common factor modulo any prime $q$ dividing $\det S$. Indeed, the minimal polynomials of $f|A_S/qA_S$ and $f|A_C/qA_C$ agree and divide $s(x)$ and $c(x)$ modulo $q$. 
The possible values of these \emph{feasible} primes are readily computed from $S$ alone, thus limiting possibilities for $S$ (and $C$). 

To reverse the process one first constructs models for $S$ and $C$ by number and lattice theory (sect. $\!\!\!\!\!\!$ \ref{sec:lattices_numberfields}) and then {\em glues} them together via the isomorphism $A_S \cong A_C$ to obtain a model for $H^2(X,\ZZ)$ together with an isometry $f$.
It is then checked that $f$ preserves a Hodge structure, represented by a suitable eigenvector of $f|H^2(X,\ZZ)\otimes \CC$. The crucial step is to check whether $f|NS(X) \otimes \mathbb{R}$ preserves a chamber representing the ample cone cut out by the nodal roots. In general it is hard to compute the (infinitely many) nodal roots, hence in \cite{mcmullen:minimum} an integer linear programming test is developed, which gives a sufficient but not necessary condition.
To resolve this uncertainty we develop a (convex) quadratic integer program refining the linear one. The quadratic integer program gives a sufficient \emph{and} necessary condition. Yet it is fast to compute (see §\ref{sec:positivity}).\\  

Let us now consider an algebraically closed field $\kappa$ of positive characteristic \mbox{$p=\cha \kappa>0$}, and let $X / \kappa$ be a supersingular K3 surface. 
Then $\NS(X)$ is an even lattice of signature $(1,21)$ and determinant $-p^{2\sigma}$ for some $1\leq \sigma \leq 10$ (the so-called {\em Artin invariant}).
As before, $f$ preserves the ample cone of $\NS(X) \otimes \mathbb{R}$ cut out by the nodal roots, as well as some extra structure (a {\em crystal}) represented by an eigenvector of $\overline{f}|A_{\NS} \otimes \kappa$.
It is proved for $p>3$ that this datum determines $(X,F)$ and any (good) datum is realized (this is more or less the content of Ogus' {\em Crystalline Torelli theorem}, see §\ref{sec:torelli}).

Thus, in our construction we have to replace $H^2(X,\ZZ)$ by $\NS(X)$ and the Torelli theorem gets a new flavor. The characteristic polynomial of $f|\NS(X)$ still factors as
\[\chi(f|\NS)=s(x) c(x)\]
where $c(x)$ is a product of cyclotomic polynomials, and the Salem and cyclotomic factors can be analogously defined as
\[S:=\ker s(f|\NS) \quad \text{and} \quad C:=\ker c(f|\NS).\]
Notice that the signature of $S$ is still $(1,d-1)$ but now that of $C$ is $(0,22-d)$.
Again $S \oplus C$ is of finite index in $\NS(X)$, but since the latter is not unimodular, there is only a partial gluing between certain subgroups (see § \ref{sec:glue})
 \[A_{S} \supseteq H_S \xrightarrow{\phi} H_C \subseteq A_{C}.\]
One can show that $pA_S \subseteq H_S$, so in this case $s(x)$ and $c(x)$ have a common factor modulo any prime dividing $|pA_S|$. In particular we take a look again at the feasible primes in Section §\ref{sec:realizability}. 

Checking whether $f$ preserves the ample cone of $\NS(X)$ is done exactly as in the complex case. The only difference is that there the failure of necessity of the linear positivity test is less severe, since often one can try a construction with a different $\NS(X)$ and hope for a positive result there. However in the supersingular case we have less freedom on $\NS(X)$ once deciding for a fixed characteristic $p$. It was for this reason that we developed the quadratic positivity test described in Theorem \ref{thm:quadratic_test}.

\subsection*{Notation}
For an even $d > 0$, $\lambda_d$ denotes the minimal Salem number of degree $d$, and $s_d(x)$ the corresponding minimal polynomial. 
Also for any integral $k > 0$, $c_k(x)$ denotes the $k$-th cyclotomic polynomial.

\subsection*{Acknowledgements}
We would like to thank Christian Lehn, S\l awomir Rams and Matthias Sch\"utt for the many stimulating discussions about this topic, and specially mention Daniel Loughran for his suggestion about the $p$-adic logarithm.

The financial support of the starting grant ERC StG 279723 ``Arithmetic of algebraic surfaces'' (SURFARI), the research training group GRK 1463 ''Analysis, Geometry and String Theory'' and the project MTM2015-69135-P of the spanish ``Ministerio de Econom\'ia y Competitividad'' is gratefully acknowledged.






\section{Lattices}
Recall that a {\em lattice} is a finitely generated free abelian group $L$ together with a non-degenerate symmetric bilinear form
$$\langle-,-\rangle: L \otimes L \longrightarrow \mathbb{Z}.$$
The signature of $L$ is the pair $\left(n_+,n_-\right)$, where $n_+$ (resp. $n_-$) is the number of positive (resp. negative) eigenvalues of the $\mathbb{R}$-bilinear extension of $\langle-,-\rangle$. A lattice is {\em hyperbolic} if $n_+ = 1$, and negative-definite (resp. positive-definite) if $n_+=0$ (resp. $n_-=0$). A lattice is called {\em even} if $\left\langle x,x\right\rangle \in 2\mathbb{Z}$ for any $x \in L$, otherwise it is called {\em odd}. The {\em orthogonal group} of $L$ is the group of {\em isometries} of $L$, that is,
$$O(L) = \left\{f: L \rightarrow L \,|\, \left\langle f(x),f(y)\right\rangle = \left\langle x,y\right\rangle \, \forall x,y \in L\right\} \subseteq GL(L).$$
As a matter of notation, if $L_1$ and $L_2$ are two lattices, the direct sum $L_1 \oplus L_2$ is meant to be the {\em orthogonal} direct sum, unless any other bilinear form is specified.\\

The non-degeneracy of the bilinear form implies that the natural map $L \rightarrow L^{\vee} = \Hom(L,\mathbb{Z})$ defined by $x \mapsto \langle x,-\rangle$ is injective, and identifies $L^{\vee}$ with the group
$$\left\{y \in L \otimes_{\mathbb{Z}} \mathbb{Q} \,|\, \langle x,y\rangle \in \mathbb{Z} \quad \forall \, x \in L\right\}.$$
The {\em discriminant group} of $L$ is defined as $A_L=L^{\vee}/L$, and naturally inherits a symmetric bilinear form
$$b_L: A_L \otimes A_L \longrightarrow \mathbb{Q}/\mathbb{Z}.$$
In case $L$ is even, there is a natural quadratic form (the {\em discriminant form}):
\[q_L: A_L \rightarrow \mathbb{Q}/2 \ZZ .\]
We say that a bilinear or quadratic form is totally isotropic on some subspace if it vanishes identically on this subspace. 
The {\em determinant} of $L$, denoted $\det(L)$, is the determinant of the Gram matrix of $\langle-,-\rangle$ with respect to any basis of $L$, and coincides up to sign with the order of the discriminant group $A_L$. More precisely
$$\det(L) = \left(-1\right)^{n_-} \left|A_L\right|.$$

\begin{definition} \label{def:ssK3lattice}
A {\em supersingular K3 lattice} is an even lattice $N$ of rank 22, signature $\left(1,21\right)$ and such that the discriminant group $A_N\cong \FF_p^{2\sigma}$, $p>2$, $\sigma \in \{1, \dots ,10\}$.
\end{definition} 

A lattice $L$ such that $A_L$ is annihilated by $p$ is called {\em $p$-elementary}.
Indefinite $p$-elementary lattices ($p\neq 2$) of rank at least $3$ are determined up to isometry by their signature and determinant. In particular supersingular K3 lattices are determined by their determinant.
To get uniqueness for $p=2$ one needs to introduce an extra invariant, namely the parity of $q_L$ \cite[Sec. 1]{rudakov_shafarevic:k3_finite_fields}.


\section{Torelli theorems for supersingular K3 surfaces} \label{sec:torelli}

In this section we recall the basic facts about supersingular K3 surfaces that are used all along the paper. In particular we introduce some versions of the Torelli theorems proved by Ogus in \cite{ogus:83}. Though crystalline cohomology plays a central role in the development and proof of these results (and even in some statements), we avoid it in order to lighten the exposition, using only the N\'eron-Severi lattice. The interested reader is referred to \cite{liedtke:lectures_supersingular,ogus:79,ogus:83} for the details.\\

Let $X$ be a K3 surface defined over an algebraically closed field $\kappa$ of characteristic $p>2$. Recall that $X$ is said to be {\em (Shioda) supersingular} if
$$\rho(X) = \rk \NS(X) = 22.$$

\begin{remark} \label{rmk:Artin-vs-Shioda}
Artin introduced in \cite{artin:supersingular} a different notion of supersingularity. Namely, a K3 surface $X$ is (Artin) supersingular if its Brauer group has infinite height, or equivalently, if the second crystalline cohomology is purely of slope 1. Due to the Igusa-Artin-Mazur inequality for varieties of finite height \cite{artin-mazur:formal_groups}, any Shioda supersingular K3 is also Artin supersingular. The converse follows from the Tate conjecture (even if the surface is not defined over a finite field, see for example \cite[Theorem 4.8]{liedtke:lectures_supersingular}). The Tate conjecture is known for K3 surfaces defined over finite fields of odd characteristic \cite{nygaard:tate,nygaard-ogus:tate,charles:tate,madapusi-pera:tate,maulik:supersingular} and has recently been announced also for $p=2$ \cite{kim-madapusi-pera:tate_2}. Therefore both definitions of supersingularity are equivalent, and from now on we thus refer to any such K3 surface simply as ``supersingular''.
\end{remark}

As said above, the N\'eron-Severi lattice $\NS(X)$ of a supersingular K3 surface $X$ is a supersingular K3 lattice. In particular, the determinant of $\NS(X)$ is of the form $-p^{2\sigma}$ for some integer $1 \leq \sigma \leq 10$, which is called the {\em Artin invariant} of $X$. Furthermore, the discriminant group is $A_{\NS(X)} \cong \mathbb{F}_p^{2\sigma}$. Moreover, the induced bilinear form on $A_{\NS(X)}$ takes values in
$$\left(\frac{1}{p}\mathbb{Z}\right)/\mathbb{Z} \cong \mathbb{F}_p$$
and is {\em non-neutral}, that is, there is no totally isotropic subspace $K \subset A_{\NS(X)}$ of dimension $\sigma = \frac{1}{2} \dim_{\mathbb{F}_p}A_{\NS(X)}$. To see this, note that neutrality would imply the existence of an even, unimodular overlattice of signature $(1,21)$. It is well known that such a lattice does not exist. \\

The aim of a Torelli-type theorem (in characteristic 0) is to characterize a variety $X$ by (part of) its Hodge structure, and maybe some extra combinatorial data. For example, the Torelli theorem for complex K3 surfaces $X$ says that $X$ is determined (up to isomorphism) by the Hodge decomposition $H^2(X,\mathbb{Z}) \otimes \CC = H^{2,0} \oplus H^{1,1} \oplus H^{0,2}$. If furthermore an ample cone in $H^{1,1}_{\mathbb{R}}$ is given (or equivalently, a chamber of effective classes), then $X$ is determined up to {\em unique} isomorphism. We now present Ogus' crystalline Torelli theorem(s) for supersingular K3 surfaces in the form most useful to us.\\
 
A positive-characteristic analogue of a Hodge structure is a {\em crystal}, associated to the crystalline cohomology groups. On a supersingular K3 surface $X$ such a crystal is determined by $\tilde{P}_X$, the kernel  of the first de Rham-Chern class map
\[c_{dR}^1\otimes \kappa:\NS(X) \otimes \kappa \rightarrow H^2_{dR}(X,\kappa).\]
Since $\kappa$ has characteristic $p$, we have
$$\NS(X) \otimes \kappa \cong \left(\NS(X) \! /p\NS(X)\right) \otimes \kappa,$$
and indeed $\tilde{P}_X$ is contained in the subspace $\left(p\NS(X)^\vee \! \!/p\NS(X) \right) \otimes \kappa$ which is clearly isomorphic to
$$\left(\NS(X)^\vee \! \! /\NS(X)\right) \otimes \kappa = A_{\NS(X)} \otimes \kappa.$$
Furthermore $\tilde{P}_X\subseteq A_{\NS(X)}\otimes \kappa$ is a ``strictly characteristic subspace'', which in general is defined as follows:

\begin{definition}\cite[Definition 3.19]{ogus:79} \label{def_charact_subspace}
Let $A$ be a $2\sigma$-dimensional $\mathbb{F}_p$-vector space equipped with a non-degenerate, non-neutral, symmetric bilinear form. Let $\Fr_{\kappa}: \kappa \rightarrow \kappa$ be the Frobenius automorphism of $\kappa$, and set
$$\psi=\id_A \otimes \Fr_{\kappa}: A \otimes \kappa \longrightarrow A \otimes \kappa.$$
A subspace $P\subseteq A \otimes \kappa$ is a \textit{characteristic} subspace if
 \begin{itemize}
    \item[(1)] $\dim_{\kappa} P=\sigma$;
    \item[(2)] $\dim_{\kappa}\left(P+\psi(P)\right)=\sigma+1$;
    \item[(3)] $P$ is totally isotropic.
  \end{itemize}
  If moreover 
  \[ \sum_{i\geq 0} \psi^i(P)=A\otimes \kappa,\]
  then $P$ is called \textit{strictly} characteristic.
\end{definition}
Note that $P$ is (strictly) characteristic if and only if  $\psi^{-1}(P)$ is. The {\em period} $P_X$ of a supersingular K3 surface $X$ is defined as
$$P_X = \psi^{-1}(\ker c_{dR}^1\otimes \kappa) = \psi^{-1}(\tilde{P}_X)$$
(see \cite{ogus:83}), and is a strictly characteristic subspace of $A_{\NS(X)}\otimes \kappa$. The following lemma follows easily from Definition \ref{def_charact_subspace}.

\begin{lemma} \label{lem_characteristic_line}
If $P \subset A \otimes \kappa$ is a strictly characteristic subspace and $\dim_{\mathbb{F}_p}A = 2\sigma$, then
$$l=P \cap \psi (P) \cap \cdots \cap \psi^{\sigma-1}(P)$$
is a line. Furthermore $P$ can be recovered as $P = l + \psi^{-1}(l) + \cdots + \psi^{-\left(\sigma-1\right)}l)$ and hence $l + \psi(l) + \cdots + \psi^{2\sigma-1}(l) = A\otimes \kappa$.
\end{lemma}

The next Theorem shows that every strictly characteristic subspace occurs as the period of some K3 surface. For this we fix the following notation: if $f: N \rightarrow M$ is an isometry of lattices, we denote by $\overline{f}: A_N \rightarrow A_M$ the induced group isomorphim (or its $\kappa$-linear extension).

\begin{theorem}[Surjectivity of the period map \cite{ogus:83}] \label{Tor1}
Given any supersingular K3 lattice $N$ and a strictly characteristic subspace $P \subset A_N\otimes \kappa$, then there is a K3 surface $X$ and an isometry $N \stackrel{\iota}{\cong} \NS(X)$, such that $\overline{\iota}(P) = P_X$.
\end{theorem}

In order to formulate a strong Torelli theorem, we need to consider the chamber structure of the positive cone in $\NS(X) \otimes \mathbb{R}$, which is analogous to that in characteristic 0. Although right now only hyperbolic lattices are needed, we recall also the chamber structure for negative-definite lattices, which will play an important role in the subsequent sections. If $L$ is an even lattice, we denote by
$$\Delta_L = \left\{\delta \in L \,|\, \delta^2=\left\langle\delta,\delta\right\rangle = -2\right\}$$
the set of {\em roots} of $L$, which is finite if $L$ is negative-definite. If $L$ is hyperbolic we set
$$V_L = \left\{x \in L\otimes\mathbb{R} \,|\, x^2 > 0 \, \text{ and } \, \left(\delta,x\right) \neq 0 \quad \forall \delta \in \Delta_L\right\},$$
which according to \cite[Proposition 1.10]{ogus:83} is an open set and each of its connected components meets $L \subset L \otimes \mathbb{R}$. These assertions still hold for negative-definite $L$ if we define
$$V_L = \left\{x \in L\otimes\mathbb{R} \,|\, x \neq 0 \, \text{ and } \, \left(\delta,x\right) \neq 0 \quad \forall \delta \in \Delta_L\right\}.$$
In both cases, the connected components of $V_L$ are called {\em chambers} of $V_L$. \\

If $L = \NS(X)$ for a supersingular K3 surface $X$, then there is exactly one chamber $\alpha_X$, the ample cone, such that a line bundle $H$ is ample if and only if $\left[H\right] \in \alpha_X$. It turns out that, together with a strictly characteristic subspace $P$, the choice of a chamber $\alpha$ in $V_L$ determines a K3 surface with ample cone $\alpha$ {\em up to unique isomorphism}. Indeed, this is a consequence of the following

\begin{theorem}\cite[Theorem II' and Theorem II'']{ogus:83}
 Let $\kappa=\overline{\kappa}$ be a field of characteristic $p>3$ and $X,Y$ supersingular K3 surfaces over $\kappa$.
 If $f:\NS(X) \rightarrow \NS(Y)$ is an isometry, then there is a unique isomorphism $F:Y \rightarrow X$ with
 $f=F^*$ provided that
 \begin{itemize}
  \item[(1)] $f(\alpha_X)=\alpha_Y$ and
  \item[(2)] $\overline{f}(P_X)=P_Y$.
 \end{itemize}
\end{theorem}

\begin{remark}
The original statements of Ogus involve $N$-marked K3 surfaces, that is, pairs $(X,\eta)$ where $\eta: N \hookrightarrow \NS(X)$ is a finite-index inclusion of a supersingular K3 lattice. This allows to consider families of surfaces with varying Artin invariant $\sigma$, which can very often happen. Indeed, it is a crucial property used in the proofs. All the definitions we have introduced above (characteristic subspaces, ample chambers, ...) carry over to this context with mild modifications. However, since we do not need this approach in our article, we have decided to avoid it for the sake of simplicity.
\end{remark}

Our main application of these results is the following immediate Corollary:

\begin{corollary} \label{cor:Strong_Torelli}
Let $\kappa=\overline{\kappa}$, $\cha \kappa>3$, $N$ a supersingular K3 lattice and $P \subset A_N \otimes \kappa$ a strictly characteristic subspace. If $f \in O(N)$ preserves a connected component of $V_N$ and $\overline{f}(P)=P$, then there is a supersingular K3 surface $X$ and an automorphism $F: X \rightarrow X$ such that $N \stackrel{\iota}{\cong} \NS(X)$, $\overline{\iota}(P)=P_X$ and $f = \iota^{-1} \circ F^* \circ \iota$.
\end{corollary}

Explicitly checking whether an isometry $f \in O(N)$ preserves some chamber of $V_N$ is not an easy task, and this is addressed in the next section.


\section{Positivity} \label{sec:positivity}

Recall from the previous section that, given an isometry $f$ of a supersingular K3 lattice $N$, we are interested in knowing whether $f$ preserves some connected component of $V_N$. To this end it is useful to consider more general lattices than only supersingular K3 lattices.\\

In what follows, let $L$ denote an even lattice which is either hyperbolic or negative-definite. Most of the definitions and several results in this section are due to McMullen \cite{mcmullen:minimum}.

\begin{definition}[Positive automorphism]
We say that an isometry $f\in O(L)$ is {\em positive} if it preserves some connected component of $V_L$.
\end{definition}

If $M \subset L$ is a sublattice of finite index, it can happen that $f|M$ is positive while $f$ is not (for example, if $M$ contains less roots than $L$, and hence then the chambers of $V_M$ are closures of unions of chambers of $V_L$). In order to emphasize this dependence on the lattice, sometimes we will say that the pair $(L,f)$ is positive.

If $L$ is hyperbolic, the {\em light cone} $\left\{x \in L\otimes \mathbb{R} \,|\, x^2>0\right\}$ has two connected components, and any positive isometry $f \in O(L)$ does not interchange them. We denote by $O^+(L)$ the subgroup of isometries with this property. If $L$ is negative-definite, then set $O^+(L) = O(L)$.\\

Since the chamber structure of $V_L$ is given by the roots of $L$, the positivity of $f$ is naturally related to its action on $\Delta_L$, and indeed there are two special kinds of roots.

\begin{definition}\cite[Obstructing and cyclic roots]{mcmullen:minimum}
Let $\delta \in \Delta_L$ be a root of $L$, and $f \in O^+(L)$ an isometry.
\begin{itemize}
\item $\delta$ is {\em obstructing} for $f$ if there is no linear form $\phi \in \Hom\left(L,\mathbb{R}\right)$ such that the bilinear form on $\ker\phi \subset L\otimes \mathbb{R}$ is negative definite and $\phi(f^i(\delta)) > 0$ for all $i\in\mathbb{Z}$.
\item $\delta$ is {\em cyclic} for $f$ if $\delta +f(\delta)+f^2(\delta)+\cdots+f^i(\delta)=0$ for some $i>0$.
\end{itemize}
\end{definition}

Obviously, any cyclic root is also obstructing. Conversely if $L$ is negative definite, all obstructing roots are cyclic.

\begin{remark}
To motivate the definition of obstructing roots, suppose that $L$ is the N\'eron-severi lattice of some projective K3 surface $X$, $f$ is induced by some automorphism $F: X \rightarrow X$, and let $h \in L$ be the class of an ample line bundle. If $\delta \in \Delta_L$ is a root, a standard computation using Riemann-Roch shows that either $\delta$ or $-\delta$ is effective. In the first case, also $f^i(\delta)$ is effective for every $i > 0$, and hence $\left\langle h,f^i(\delta)\right\rangle > 0$ for every $i > 0$. Thus, the linear form $\phi=\left\langle h,-\right\rangle$ shows that $\delta$ cannot be obstructing (the negative-definiteness on $\ker\phi$ follows from the Hodge-index theorem). In case $-\delta$ is effective, then $\phi=-\left\langle h,-\right\rangle$ leads to the same conclusion. Therefore, an obstructing root is indeed an obstruction to the existence of an ample line bundle on $X$.
\end{remark}

It was proved by McMullen \cite[Theorem 2.2]{mcmullen:minimum} that an isometry $f \in O^+(L)$ is positive if and only if it has no obstructing roots. In the same work \cite[Section 3]{mcmullen:minimum}, McMullen developed a method to detect obstructing roots that can be summarized as follows. Denote by $\tilde{c}(x)$ the part of the characteristic polynomial $c(x)$ of $f$ which is coprime to $(x-1)$. First one looks at cyclic roots, which by definition lie in the kernel of $\tilde{c}(f)$ which is negative-definite, and thus its roots are easily computed.

We can therefore assume that $L$ is hyperbolic and $f$ has no cyclic root, for otherwise it is not positive. Let $a = f + f^{-1}$ and $A = \mathbb{R}\left[a\right] \subset \End_{\mathbb{R}}\left(L \otimes \mathbb{R}\right)$. Given any $x \in L$, let $\psi_x: A \rightarrow \mathbb{R}$ be the {\em pure state} defined by $\psi_x(a)=\left\langle a(x),x\right\rangle$, and consider the lattice of {\em mixed states} $M \subset \Hom_{\mathbb{R}}\left(A,\mathbb{R}\right)$ spanned by $\left\{\psi_x \,|\, x \in L\right\}$. If $e_1,\ldots,e_n$ is a $\mathbb{Z}$-basis of $L$, $M$ turns out to be generated by the $\psi_{e_i}$ and $\psi_{e_i+e_j}$ \cite[Proposition 3.2]{mcmullen:minimum}. By construction, $a$ diagonalizes with real eigenvalues, which we denote by $\tau_1 > \tau_2 > \ldots > \tau_r$. Then we define $V_i =\ker\left(a-\tau_i\Id\right) \subset L\otimes\mathbb{R}$, obtaining an $f$-invariant orthogonal decomposition $L\otimes\mathbb{R} = \sum_{i=1}^r V_i$. Let $p_1,\ldots,p_r$ be the corresponding projections, so that $p_i^2 = p_i$, $a\circ p_i = \tau_ip_i$ and $\sum_{i=1}^r p_i = 1$.
 
With all these ingredients the following integer linear programming problem can be defined: let
\begin{equation} \label{eq:linear_prog_test}
\mu(f) = \max\left\{\psi(1) \,|\, \psi \in M, \psi(p_1) < 0 \, \text{ and } \, \psi(p_i) \leq 0 \quad \forall \, i>1\right\}.
\end{equation}
Note that by construction $\psi(1) \in 2 \mathbb{Z}$ for any $\psi \in M$, hence in any case $\mu(f) \leq -2$. If $\mu(f)<-2$ (and $f$ has no cyclic roots), McMullen proved that $f$ is positive \cite[Theorem 3.3, Linear positivity test]{mcmullen:minimum}. Note that the condition $\mu(f)<-2$ is only sufficient for the positivity of $f$, but not necessary. Indeed, there are examples of positive automorphisms with $\mu(f)=-2$ (see \cite{mcmullen:minimum}). The reason for this failure is that the maximizing $\psi$ is not necessarily a pure state $\psi_\delta$ for some $\delta \in L$. Instead, it might be a linear combination of pure states.

In order to have a necessary and sufficient condition, we have obtained the following result.

\begin{theorem}[Quadratic positivity test] \label{thm:quadratic_test}
Fix $y \in V_1$ with $y^2>0$ and for any $\psi \in M$ set
  \[B_\psi=\{x \in L \otimes \RR \mid \left\langle x,y\right\rangle\leq 0, \, \left\langle x,f(y)\right\rangle\geq 0, \text{ and } p_i(x)^2=\psi(p_i), \text{ for all } i\}.\]
Then $f$ is positive if and only if it has no cyclic roots and for every optimal solution $\psi$ of the linear positivity test (\ref{eq:linear_prog_test}) with $\psi(1)=-2$, the compact set $B_\psi$ contains no integral points, i.e. $B_\psi \cap L = \emptyset$.  
\end{theorem}
\begin{proof}
  By the previous discussion it is enough to show that an optimal solution $\psi$ with $\psi(1)=-2$ is a pure state if and only if $B_\psi\cap L\neq \emptyset$. 

  
  Suppose that $\psi$ is pure, i.e., $\psi=\psi_\delta$ for some $\delta \in L$.
  Observe that for any $i=1,\ldots,r$ it holds
  \[\psi_\delta(p_i)=\langle p_i(\delta),\delta\rangle=p_i(\delta)^2.\]
  It remains to check the two inequalities in the definition of $B_{\psi}$. Since $\langle x,y\rangle =\langle p_1(x),y\rangle$, we only need to consider the situation in $V_1$.
  It is an indefinite plane, and there the points of length $\psi(p_1)$ form a (non-compact) hyperbola whose asymptotes $\left\{x^2=0\right\}$ are the eigenspaces of $f|V_1$. Furthermore $f$ acts by translation along this hyperbola.
  Since $f$ is an isometry and commutes with $\RR [a]$, we get that $\psi_\delta=\psi_{\pm f^n(\delta)}$. Hence after replacing $\delta$ by a suitable $\pm f^n(\delta)$, we can assume that $\delta \in B_\psi$.
    
  We now turn to the compactness of $B_\psi=\prod_i (V_i \cap B_{\psi})$.
  Recall that $V_i$ is negative definite for $i \geq 2$, hence 
  $(V_i \cap B_{\psi})$ is a (compact) sphere of radius $\sqrt{-\psi(p_i)}$.
  Since $y^2>0$, the lines $y^\perp \cap V_1$ and $f(y)^\perp \cap V_1$ intersect each connected component of the hyperbola in a single point. Then $V_1 \cap B_{\psi}$ is the path along one connected component of the hyperbola between these two points, which is thus compact. 
  
\begin{figure}[!htbp] \label{fig:positivity_plot}
\includegraphics[width=8cm]{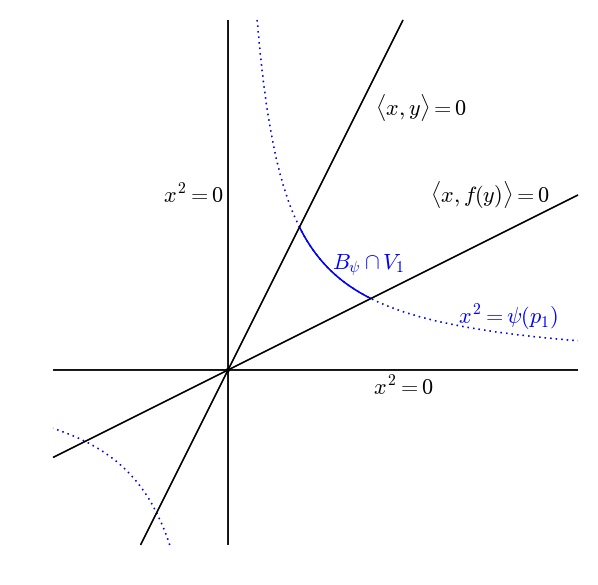}
\caption{Fundamental domain in $V_1$.}
\end{figure}
\end{proof}

\begin{remark}
For practical applications we compute the integral points of the convex hull Conv$(B_\psi)$
with SCIP\cite{scip} and cplex\cite{cplex} and check which of them are roots. Depending on the rank, computation times vary between seconds and a few minutes.  
\end{remark}

The following corollary shows that we do not need to care too much about the isometry on any negative-definite invariant sublattice as long as it is positive.
\begin{corollary}\label{coro:positivity_finite_part_not_important}
Let $S\oplus C \hookrightarrow L$ be a primitive extension (see forthcoming Definition \ref{def:primitive_extension}) of a hyperbolic lattice $S$ and a negative definite lattice $C$. Let $f_S \in O^+(S), f_1,f_2 \in O(C)$ be positive automorphisms such that $f_S \oplus f_i, i=1,2$ extends to $L$. Then 
$(L,f_S\oplus f_1)$ is positive if and only if $(L,f_S \oplus f_2)$ is. 
\end{corollary} 
\begin{proof}
Since the $f_i$ are of finite order, we can find $n \in \mathbb{N}$ such that $(f_S \oplus f_1)^n=(f_S \oplus f_2)^n$. To finish the proof note that if an obstructing root is not cyclic it stays obstructing for all powers of an isometry. 
\end{proof}



\section{Lattices in number fields}  \label{sec:lattices_numberfields}
Our goal is to construct a lattice together with an isometry of given characteristic polynomial. First we consider the case when the characteristic polynomial is irreducible, and the reducible case is treated in the next section. The basic construction we recall in this section is due to McMullen \cite{mcmullen:minimum}, where the reader is referred for more details and proofs. \\

A pair $(L,f)$ of a lattice $L$ and an isometry $f$ of $L$ with characteristic polynomial $p(x)\in \mathbb{Z}[x]$ is called a $p(x)$-lattice. Two $p(x)$-lattices $(M,f)$ and $(N,g)$ are isomorphic if there is an isometry $\phi: M \rightarrow N$ with $\phi \circ f=g \circ \phi$.
Given a $p(x)$-lattice $(L,f)$ and $a \in \mathbb{Z}[f+f^{-1}]$, we obtain a new symmetric bilinear form on $L$ by setting
\[\left\langle g_1 , g_2 \right \rangle_a=\left\langle ag_1 , g_2 \right \rangle.\]
The lattice $L$ equipped with this new product is called the \textit{twist} of $L$ by $a$ and is denoted by $(L(a),f)$. Note that the twist of an even lattice stays even.\\
\\
A polynomial $p(x)\in \mathbb{Z}[x]$ of degree $2d$ is called {\em reciprocal} if $x^{2d}p(x^{-1})=p(x)$, or equivalently if its coefficients form a palindrome. 
Associated to such $p(x)$ is its trace polynomial $r(y)\in\mathbb{Z}[y]$, defined by the equality \[p(x)=x^d r(x+x^{-1}).\]
From now on we assume that $p(x)$ is irreducible. Then $K:=\mathbb{Q}[x]/p(x)$ is a quadratic field extension of $k:=\mathbb{Q}[y]/r(y)$. The principal $p(x)$-lattice $(L_0,f)$ is the abelian group $L_0=\mathbb{Z}[x]/p(x) \subseteq K$, equipped with the bilinear pairing
\[\left\langle g_1(x),g_2(x) \right\rangle= Tr^K_\mathbb{Q}\left(\frac{g_1(x)g_2(x^{-1})}{r'(x+x^{-1})}\right)\]
(where $r'(y)$ the formal derivative of $r(y)$) and the isometry $f$ given by multiplication with $x$. Its characteristic polynomial is of course $p(x)$.
We note that the bilinear form is even and its determinant is given by $|\det L_0|=|p(1)p(-1)|$.\\
We call an irreducible reciprocal polynomial $p(x)\in\mathbb{Z}[x]$ \emph{simple} if
$\mathbb{Z}[x]/p(x)=\mathcal{O}_K$ is the full ring of integers of $K$, the field $K$ has class number one, and $|p(1)p(-1)|$ is square free \cite{mcmullen:minimum}.

Examples of simple reciprocal polynomials include the minimal Salem polynomials in degrees up to $22$ as well as the cyclotomic polynomials $c_n$ of degree up to $20$ if $n\neq 2^k$ (cf. \cite{mcmullen:minimum}). Moreover, if $p(x)$ is simple reciprocal, then every $p(x)$-lattice is isomorphic to a twist $L_0(a)$ of the principal $p(x)$-lattice with $a \in \mathcal{O}_k$ \cite[Theorem 5.2]{mcmullen:minimum}.
The discriminant group of a twist $L_0(a)$ is easily controlled by the norm of $a$. 
\begin{lemma}\label{lem:dual is fractional ideal}
 Let $p(x)$ be a simple reciprocal polynomial. Then
 the dual lattice $L_0^\vee$ of the principal $p(x)$-lattice is a fractional ideal. Moreover
 \[L_0^\vee= \tfrac{1}{t}\mathcal{O}_K\]
 for some $t \in \mathcal{O}_K$ and
 \[A_{L_0(a)} \cong L_0^\vee/aL_0 \cong \mathcal{O}_K/at\mathcal{O}_K,\]
 are isomorphic as $\mathcal{O}_K$-modules.
 In particular
 \[\left|\det L_0(a)\right|=\left|\det L_0 \cdot N^K_\mathbb{Q}(a)\right|.\]
\end{lemma}
\begin{proof}
 By simplicity of $p(x)$, $\mathbb{Z}[f]=\mathcal{O}_K$. Recall that $f$ acts on $L_0^\vee$ by $\QQ$-linear extension, hence $L_0,L_0^\vee$ and $A_{L_0}$ are $\mathbb{Z}[f]=\mathcal{O}_K$-modules.  
 Notice that $L_0^\vee$ is finitely generated and contained in $L_0 \otimes \QQ=K$, and in particular it is a fractional ideal. 
 Again by simplicity $\mathcal{O}_K$ has class number one, hence all fractional ideals are principal. Since $L_0=\mathcal{O}_K\subseteq L_0^\vee$, we can find $t\in \mathcal{O}_K$ such that \[L_0^\vee= \tfrac{1}{t}\mathcal{O}_K.\]
 For the statement about the discriminant group, let $v\in L_0(a)^\vee$. This is equivalent to $\left\langle av,L_0\right\rangle\subseteq \mathbb{Z}$, which in turn means $v\in \tfrac{1}{a}L_0^\vee$. This gives the isomorphism
 \[A_{L_0(a)}=\left(\tfrac{1}{a}L_0^\vee\right)\! /L_0 \cong L_0^\vee/aL_0,\]
from which the last formula follows at once.
\end{proof}

We see that twisting by a unit leaves the discriminant group unchanged, but it may change the discriminant form and the signature (for a precise description of these changes we refer to \cite{mcmullen:siegel_disk}). However, twisting by the square $u^2$ of a unit $u \in \mathcal{O}_k^\times$ results in an isomorphic $p(x)$-lattice:
\[L(u^2)\cong L, \quad x \mapsto ux.\]
This implies that there are only finitely many non-isomorphic $p(x)$-lattices of given discriminant. 


\section{Gluing lattices and isometries} \label{sec:glue}
Now, let $(L,f)$ be a $p(x)q(x)$-lattice with $p(x),q(x) \in \mathbb{Z}[x]$ coprime polynomials. 
Then $\ker p(f) \oplus \ker q(f)$ is a finite index subgroup of $L$. In this section we see how to construct $L$ from the smaller parts $\ker p(f)$ and $\ker q(f)$. The theory of primitive extensions and their relation to discriminant forms, which is the base of our construction, was initiated by Nikulin in \cite{nikulin:quadratic_forms}. From now on all lattices are supposed to be even.\\

We call an embedding of lattices $N\hookrightarrow L$ primitive if $L/N$ is torsion free.
Primitive sublattices arise as kernels of endomorphisms and also in geometry, such as 
$\NS(X)$ or the transcendental lattice $T(X)$ in $H^2(X,\mathbb{Z})$ of a complex $K3$ surface.
\begin{definition}[Primitive extension \cite{nikulin:quadratic_forms}] \label{def:primitive_extension}
Let $M$ and $N$ be two lattices. A {\em primitive extension} of $M$ and $N$ is an overlattice $M \oplus N \hookrightarrow L$ (of the same rank), such that $M$ and $N$ are primitive sublattices of $L$. 
\end{definition}
Primitive extensions are determined by a glue map $\phi$ defined on certain subgroups 
 \[A_{M} \supseteq H_M \xrightarrow[\phi]{\sim} H_N \subseteq A_{N},\]
 with the extra condition that $q_M=-q_N \circ \phi$.
Indeed, given $\phi$ as above, we define \emph{the glue} 
\[H:=\{x+\phi(x) |x\in H_M\} \subseteq A_{M} \oplus A_{N}\] 
of the primitive extension as the graph of $\phi$.
By construction, $H$ is a totally isotropic subspace of $A_M \oplus A_N$. Hence we can define an integral lattice $L = M \oplus_{\phi} N$ via
\begin{equation}\label{eq:glue H isos}
L/(M\oplus N) =H \cong H_M\cong H_N.
\end{equation}
It is not hard to see that $A_L=H^\perp/H$. Conversely, if $M \oplus N \hookrightarrow L$ is a primitive extension and $p_M: L^\vee \rightarrow M^\vee$ and $p_N: L^\vee \rightarrow N^\vee$ are the natural projections, there is a natural isomorphism $H_M:=p_M(L)/M \longrightarrow H_N:=p_N(L)/N$. \\

If $L$ is unimodular, it is well known that
\[A_{M}=H_M \xrightarrow[\phi]{\cong} H_N=A_{N}.\]
For example $A_{T(X)} \cong A_{\NS(X)}$ for a K3 surface $X$ over $\mathbb{C}$.
For more general lattices $L$, there is the following constraint on the size of the glue:

\begin{lemma}\label{lem:glue estimate}
 \[|A_{N}/H_N| \cdot |A_{M}/H_M| = \det L \]
\end{lemma}

\begin{proof}
Divide the standard formula 
 \[\det M \det N = [L:M \oplus N]^2 \det L \]
by $[L:M\oplus N]^2$ and use the isomorphisms (\ref{eq:glue H isos}).
\end{proof}

We now prove some technical results that are used often in the sequel.

\begin{lemma} \label{lem:structure_glue_quotient}
Let $N \hookrightarrow L$ be a primitive embedding.
Then there is a surjection $A_L \twoheadrightarrow A_N/H_N$.
\end{lemma}
\begin{proof}
We have the following induced diagram with exact rows
\begin{displaymath}
\xymatrix{
0 \ar[r] & L \ar[r] \ar@{->>}[d] & L^{\vee} \ar[r] \ar@{->>}[d] & L^{\vee}/L = A_L \ar[r] \ar@{-->>}[d] & 0 \\
0 \ar[r] & p_N(L) \ar[r] & N^{\vee} \ar[r] & N^{\vee}/p_N(L) \cong A_N/H_N \ar[r] & 0
}
\end{displaymath}
where the primitivity of $N \hookrightarrow L$ gives the surjectivity of the central vertical arrow. The commutativity of the diagram then implies the desired surjection.
\end{proof}

\begin{corollary} \label{cor:structure_glue_quotient}
Let $M \oplus N \hookrightarrow L$ be a primitive extension and $q$ a prime number. 
If $L$ is $q$-elementary, then the quotient $A_N / H_N$ is an $\mathbb{F}_q$-vector space.
\end{corollary}

An isometry $f_M \oplus f_N$ defined on $M \oplus N$ extends to $L$ if $\overline{f_M}: A_M \rightarrow A_M$ preserves $H_M$, $f_N$ preserves $H_N$ and $\phi \circ \overline{f_M}=\overline{f_N} \circ \phi$. This imposes compatibility conditions on the characteristic polynomials $\chi_M$ and $\chi_N$ of the two isometries. In particular, the following result, proved originally by McMullen for unimodular primitive extensions, holds also for arbitrary ones.

\begin{theorem} \label{thm:resultant}
With the above notations, if $f_M\oplus f_N$ extends to $L$, then any prime number dividing $|H|$ also divides the resultant $\res(\chi_M,\chi_N)$.
\end{theorem}
\begin{proof}
The proof of \cite[Theorem 4.3]{mcmullen:minimum} for the unimodular case works verbatim for general primitive extensions.
\end{proof}

Concerning sufficient conditions for the isometry $f_M \oplus f_N$ to extend to $L$, McMullen obtained some results when the discriminant groups are direct sums of $\mathbb{F}_q$-vector spaces (\cite[Theorem 3.1]{mcmullen:entropy_and_glue}
), a situation that in good cases can be achieved by twisting (\cite[Theorem 4.3]{mcmullen:entropy_and_glue}
).\\

For later use we close this section with the following 
\begin{proposition}
 Let $L$ be a $q$-elementary lattice and $f \in O(L)$ an isometry. Then the characteristic polynomial $\chi_{\overline{f}|A_L}(x) \in \mathbb{F}_q\left[x\right]$ divides the reduction of $\chi_{f|L}(x)$ modulo $q$.
\end{proposition}
\begin{proof}
Consider the following exact sequence of $\mathbb{F}_q$-vector spaces. 
 \begin{diagram}
  0 &\rTo &L/qL^\vee &\rTo & L^\vee/qL^\vee & \rTo &L^\vee/L &\rTo &0\\
 \end{diagram}
It is compatible with the action of $f$ on each part. Thus the splitting of this sequence is compatible with $f$.
To conclude the proof recall that $\chi_{f|L^\vee}=\chi_{f|L}$ and notice that $\chi_{\overline{f}|\left(L^\vee/qL^\vee\right)}\equiv \chi_{f|L^\vee} \mod q$. 
\end{proof}


\section{Realizability Conditions} \label{sec:realizability}
Now, we connect our knowledge of the crystalline Torelli theorem and gluing to study automorphisms of positive algebraic entropy on supersingular K3 surfaces defined over an algebraically closed field $\kappa$ of characteristic $p\geq 5$. From now on, let $X$ be such a supersingular K3 surface with N\'eron-Severi lattice $\NS$. Let $F\in \Aut(X)$ be an automorphism, $f = F^*: \NS \rightarrow \NS$ the corresponding isometry of $\NS$, and $\overline{f}: A_{\NS} \rightarrow A_{\NS}$ the induced isometry of the discriminant group (or its $\kappa$-linear extension).

For a complex K3 surface $X$ the minimal polynomial of $f|T(X)$ is irreducible over $\QQ$ and hence its characteristic polynomial is a perfect power. The following proposition shows that a similar statement holds for supersingular K3 surfaces if we replace $T(X)$ by $A_{\NS(X)}$.
\begin{proposition}\label{prop:perfect_power}
The minimal polynomial of $\overline{f}$ is irreducible. In particular its characteristic polynomial
 $\chi_{\overline{f}} \in \mathbb{F}_p[x]$ is a perfect power. 
\end{proposition}
\begin{proof}
Note that $\overline{f}$ preserves the period $P_X \subset A_{\NS} \otimes \kappa$ of $X$. Since $\overline{f}$ and the semilinear automorphism $\psi=\id \otimes \Fr_{\kappa}: A_{\NS} \otimes \kappa \rightarrow A_{\NS} \otimes \kappa$ commute, the line (cf. Lemma \ref{lem_characteristic_line})
 \[l_X=P_X \cap \psi (P_X) \cap \cdots \cap \psi^{\sigma-1}(P_X)\] 
is preserved by $\overline{f}$ as well. But $\sum_i \psi^i(l_X)=A_{\NS} \otimes \kappa$, and hence we can find a basis of eigenvectors of $\overline{f}|A_{\NS}\otimes \kappa$ on which $\psi$ acts transitively.
This shows that the eigenvalues are roots of a single irreducible polynomial in $\mathbb{F}_p[x]$, the minimal polynomial of $\overline{f}$.  
\end{proof}

Assume from now on that $F$ has positive entropy and recall (see \cite[Section 3]{mcmullen:siegel_disk}) that the characteristic polynomial of $f$ factors as 
\[\chi_f=s(x)c(x)\]
where $s(x)$ is a Salem polynomial and $c(x)$ is a product of cyclotomic polynomials.
Morever the sublattices
\[S=\ker s(f) \qquad \mbox{ and } \qquad C=\ker c(f).\]
are respectively hyperbolic and negative definite. In particular, the action of $f$ on $\NS \otimes \mathbb{C}$ is semisimple, i.e., the minimal polynomial is separable. 
The inclusion
\[S\oplus C \hookrightarrow \NS\]
is a primitive extension of $S$ and $C$. By Theorem \ref{thm:resultant}, gluing can occur only over the primes $q \mid \res(s,c)$. We call such primes {\em feasible for $c$ and $s$}.

\begin{corollary}\label{cor:p_not_feasible}
If $ \cha \kappa=p$ is not feasible for $c$ and $s$, then either $A_{S,p}=0$ or $A_{C,p}=0$ (where $A_{S,p}$ resp. $A_{C,p}$ denotes the $p$-primary part of $A_S$ resp. $A_C$).
\end{corollary}
\begin{proof}
 If $p$ is not feasible for $c$ and $s$, then $p\nmid res(s,c)$ and hence we cannot glue over $p$, i.e. $A_{\NS}=A_{\NS,p}=A_{S,p} \oplus A_{C,p}$.
 In particular $$\chi_{\overline{f}|A_{S,p}} \cdot \chi_{\overline{f}|A_{C,p}}=\chi_{\overline{f}|A_{\NS}},$$ which is a perfect power by Proposition \ref{prop:perfect_power}. But $\chi_{\overline{f}|A_{S,p}} \mid \overline{s(x)}$ and $\chi_{\overline{f}|A_{C,p}} \mid \overline{c(x)}$ are coprime. 
 This is only possible if $A_{S,p}=0$ or $A_{C,p}=0$.  
\end{proof}

Note that a priori we only know $s(x)$, the minimal polynomial of the Salem number we want to realize as the exponential of the entropy of $f$, but there are many possibilities for $c(x)$. As a first constraint, we know that $c(x)$ is a product of cyclotomic polynomials $c_k(x)$ of total degree $22 - \deg s(x)$. Thus we say that a prime number $q \in \mathbb{Z}$ is {\em feasible} (for $s(x)$) if
\[q \mid \prod_{\varphi(k)\leq 22-d} \res(s,c_k),\]
or equilvalently, if the reduction $\overline{s}(x) \in \mathbb{F}_p\left[x\right]$ has a cyclotomic factor of degree at most $22 - \deg s$. In particular we can only glue over the feasible primes. \\

The following Theorem gives a list of necessary conditions for an isometry on $S$ to admit an extension to $\NS$, and in particular further restrictions on the cyclotomic factor.
We denote by $D(n)$ the minimum $D\geq 0$ such that $\mathbb{Z}^D$ has an automorphism of order $n$. Note that $D(1)=0$, $D(2)=1$ and $D(n)=D(n/2)$ if $n\equiv 2 \mod 4$. 
In any other case we have 
\[D(p_1^{e_1}\cdot\ldots\cdot p_s^{e_s})=\sum \varphi(p_i^{e_i})\]
for the prime decomposition of $n$. 

\begin{theorem}\label{thm:realizability_conditions}
Let $f$, $s(x)$ and $S$ be defined as above. Then:
 \begin{itemize}
  \item[(1)] The determinant of $S$ is divisible only by the feasible primes (for $s$) and the characteristic $p$.
  \item[(2)] The order $n$ of $\overline{f}$ on the subgroup $pA_{S} \subseteq A_S$ satisfies $D(n)\leq 22- \deg(s)$.
  \item[(3)] There is a product of distinct cyclotomic polynomials $\mu(x)$ with $\deg{ \mu (x)} \leq 22-\deg s(x)$ and $\mu(\overline{f}|pA_{S})=0$.
  \item[(4)] $f|S$ is positive. 
 \end{itemize}
\end{theorem}
\begin{proof}
\begin{itemize}
\item[(1)] By Corollary \ref{cor:structure_glue_quotient}, we have $pA_S \subseteq H_S$, while Theorem $\ref{thm:resultant}$ implies that only feasible primes divide $|H_S|$.
\item[(2)] The order $n$ of $f|C$ satisfies $D(n)\leq 22- \deg(s)$ and it is a multiple of the order of $\overline{f}|H_C$, which in turn is a multiple of the order on $pA_S \subset H_S \cong H_C$.
\item[(3)] The isomorphism $H_S\cong H_C$ is compatible with $f$. Let $\mu$ be the minimal polynomial of $f|C$. It is a product of different cyclotomic polynomials because $f$ acts semisimply on $\NS$. Then $\mu(f)$ vanishes on $C$ and consequently on $A_C$. By compatibility of the action it vanishes on $H_C\cong H_S\supseteq p A_S$ as well. 
\item[(4)] $f$ is itself positive (on $\NS$), and therefore so is any restriction.
\end{itemize}
\end{proof}

The following is a partial converse to Proposition \ref{prop:perfect_power}, which is enough to realize all minimal Salem numbers as entropies of automorphisms of supersingular K3 surfaces.

\begin{theorem}\label{thm:realized}
 Let $N$ be a supersingular K3 lattice of determinant $-p^{2\sigma}$, $p\geq 5$, $1 \leq \sigma \leq 10$. Let
 $f\in \mathcal{O}(N)$ be such that
 \begin{itemize}
  \item[(1)] $f$ is positive, and
  \item[(2)] the characteristic polynomial $\chi_{\overline{f}|A_{N}}$ is irreducible.
 \end{itemize}
 Then there is a supersingular K3 surface $X$, an isometry $\eta: N \rightarrow \NS(X)$ and $F\in \Aut(X)$ such that
 \[f=\eta^{-1} F^*\eta.\]
\end{theorem}

\begin{proof}
 Choose an eigenvector $e \in A_{\NS}\otimes \kappa$ of $\overline{f}$ with eigenvalue $\alpha$. Since the characteristic polynomial is irreducible,
 $\left\{\psi^i(e);0\leq i \leq 2\sigma-1\right\}$ is an eigenbasis of $\overline{f}|A_{N}\otimes \kappa$.
 We claim that $P=\kappa e+\kappa\psi(e)+\ldots+\kappa\psi^{\sigma-1}(e)$ is a strictly characteristic subspace. 
 The only non-trivial observation is that $P$ is totally isotropic. Indeed, for any $0 \leq m \leq n < \sigma$ it holds
 \[\left \langle \psi^n e,\psi^m e \right\rangle =\left\langle f(\psi^n e),f(\psi^m e) \right\rangle =\alpha^{p^{n}+p^m} \! \left\langle \psi^n e,\psi^m e \right\rangle\]
 Since $\chi_{\overline{f}|A_{N}}$ is irreducible, the order of Frobenius on the eigenvalues is $2\sigma$. In particular $\alpha^{p^n+p^m} \neq 1$ if $0\leq n-m \leq \sigma-1$, as otherwise
 \[\Fr_{\kappa}^{n-m}(\alpha)=\alpha^{-1} \implies \sigma \mid (n-m).\]
 Finally, since $P$ is preserved by $f$ by construction, Corollary \ref{cor:Strong_Torelli} gives us the final statement.
\end{proof}

Note that in the setting of the previous theorem, we get only finitely many such K3 surfaces up to isomorphism. If one only requires the minimal polynomial $\mu_{\overline{f}|A_{N}}$ to be irreducible, then the theorem remains true and instead one expects a family \cite{brandhorst:phd}. Since we do not need this case and its proof is more involved, it is omitted. 

We close this section with a finiteness result on realizable twists of a given lattice. 

\begin{proposition}\label{prop:finite number of  twists}
 Let $s(x)$ be a simple Salem polynomial and $L_0$ the principal $s(x)$-lattice. 
 Then only a finite number of twists $L_0(a)$, $a\in \mathcal{O}_k$ is realizable as Salem factor of an automorphism of a 
 supersingular K3 surface in a fixed characteristic $p$.
\end{proposition}
\begin{proof}
Since the associates of $a\in \mathcal{O}_k$ define only finitely many non-isomorphic $s(x)$-lattices, it suffices to bound the possible prime factorizations of $a$ in $\mathcal{O}_K$ such that
\[L_0(a)\cong \ker s(f|\NS)\]
where $f \in \Aut(X)$ of a supersingular K3 surface $X$ in characteristic $p$.
According to Lemma \ref{lem:dual is fractional ideal}, we can find an ideal $I<\mathcal{O}_K$ such that
\[A_{L_0(a)} \cong  \mathcal{O}_K / I\]
as $\mathbb{Z}\left[f\right]=\mathcal{O}_K$-modules.
By Theorem \ref{thm:realizability_conditions} $|A_{L_0(a)}|$ is divisible at most by the feasible primes and $p$.
Thus only finitely many prime ideals $\mathfrak{p}$ are possible divisors of $I$ and hence of $a\mathcal{O}_K$.
By Theorem \ref{thm:realizability_conditions} the order of $f|pA_{L_0(a)}$ is bounded. 
We view $pA_{L_0(a)}$ as an ideal of $\mathcal{O}_K/I$. Using the Chinese remainder theorem we can reduce to the case that
$I=\mathfrak{p}^l$ has a single prime divisor and 
\[pA_{L_0(a)} \cong p\left(\mathcal{O}_K/\mathfrak{p}^l\right)=\mathfrak{p}^e/\mathfrak{p}^l \cong \mathcal{O}_K/\mathfrak{p}^{l-e}\]
for some fixed $e\in \mathbb{N}$ independent of $l$. 
Looking at Lemma \ref{lem:exponential growth} below we see that the order of $f$ on $\mathcal{O}_K/\mathfrak{p}^{l-e}$ grows exponentially in $l$, proving that $l$ is bounded above as wanted. 
\end{proof}

In the above proof we needed to control the order of an automorphism of $\mathcal{O}_K/\mathfrak{p}^n$. We may replace $\mathcal{O}=\mathcal{O}_K$ by its completion $\hat{\mathcal{O}}$ at $\mathfrak{p}$ since $\mathcal{O}/\mathfrak{p}^l\mathcal{O} \cong \hat{\mathcal{O}}/\hat{\mathfrak{p}}^l\hat{\mathcal{O}}$. We can thus use the following elementary results from the theory of $p$-adic numbers. Let $K$ be a finite extension of $\mathbb{Q}_p$, $\mathcal{O}$ its ring of integers with maximal ideal $\mathfrak{p}$, and $\nu_{\mathfrak{p}}$ the corresponding normalized valuation. Let $e$ be the ramification index of $p$, that is, $p\mathcal{O}=\mathfrak{p}^e$.

\begin{proposition}\cite[II Prop. 3.10 and 5.5]{neukirch:ant} \label{prop:neukirch exp iso}
Let $U^{(n)}=1+\mathfrak{p}^n \subset \mathcal{O}^\times$. Then
 \[\mathcal{O}^\times/U^{(n)} \cong (\mathcal{O}/\mathfrak{p}^n)^\times\]
 for $n\geq 1$. Furthermore, the power series
 \[\exp(x)=1+x+\frac{x^2}{2!}+\frac{x^3}{3!}+\cdots \quad \mbox{and} \quad \log(1+x)=x-\frac{x^2}{2}+\frac{x^3}{3}-\cdots,\]
 yield, for $n>\frac{e}{p-1}$, two mutually inverse isomorphisms
 \[\mathfrak{p}^n \rightleftarrows U^{(n)}.\]
\end{proposition}

\begin{lemma}\label{lem:order in 1+p}
 In the setting of the preceding proposition let $f\in U^{(n)}\setminus U^{(n+1)}$.
 For $l\geq n >\frac{e}{p-1}$ the order of $f$ in $U^{(n)}/U^{(l)}$ is $p^{\left \lceil{\frac{l-n}{e}}\right \rceil}$.
\end{lemma}
\begin{proof}
 By assumption $l\geq n>\frac{e}{p-1}$ so the order of $f$ in $U^{(n)}/U^{(l)}$ is that of
 $\log(f)$ in $\mathfrak{p}^n/\mathfrak{p}^l$. Write $f=1+z$ for $z\in\mathfrak{p}^n$.
 It follows from the proof of Proposition \ref{prop:neukirch exp iso} that
 \[\nu_\mathfrak{p}(\log(1+z))=\nu_\mathfrak{p}(z)=n.\]
 Note that $kz \equiv 0 \mod \mathfrak{p}^l$ if and only if $l\leq\nu_\mathfrak{p}(kz)=\nu_\mathfrak{p}(k)+n$ if and only if
 $l-n \leq \nu_\mathfrak{p}(k)$. The smallest such $k \in \mathbb{N}$ is the order of $f$ in $\left(\mathcal{O}/\mathfrak{p}^l\right)^\times$. It equals
 $p^{\left\lceil{\frac{l-n}{e}}\right\rceil}.$ 
\end{proof}

\begin{lemma}\label{lem:exponential growth}
For $f\in \mathcal{O}^\times$ denote by $o(f,l)$ the order of $\overline{f} \in (\mathcal{O}/\mathfrak{p}^l)^\times$.
If $l\geq n>\frac{e}{p-1}$, where $f^{o(f,1)}\in U^{(n)}\setminus U^{(n+1)}$, then
\[o(f,l)=o(f,1)p^{\left \lceil{\frac{l-n}{e}}\right \rceil}.\]
\end{lemma}

\begin{proof}
Let $\alpha =o(f,1)$. With $\alpha \mid o(f,l)$, we get that
$o(f^\alpha,l)=\lcm(o(f,l),\alpha)/\alpha=o(f,l)/\alpha$.
Thus, after replacing $f$ by $f^\alpha$, 
the conditions of Lemma $\ref{lem:order in 1+p}$ are fulfilled and the order of $f$ is $\alpha p^{\left\lceil{\frac{l-n}{e}}\right\rceil}$ as claimed. 
\end{proof}

\section{Realized Salem numbers}
We summarize now the strategy that can be followed to realize a (simple) Salem number $\lambda$ as the exponential of the entropy of an automorphism of a supersingular K3 surface. Although it is basically the same strategy that McMullen follows in \cite{mcmullen:minimum}, we explicitly include it since the positive characteristic $p$ introduces some new features that have to be controlled at some steps.

Let $s(x)$ be the minimal polynomial of $\lambda$, of degree $d$, and $r(y)$ be the corresponding trace polynomial, then the field $K=\mathbb{Q}[x]/s(x)$ is a quadratic extension of $k=\mathbb{Q}[y]/r(y)$. If $s$ is simple, then $\mathcal{O}_K=\mathbb{Z}[x]/s(x)$ has class number one, and furthermore $\mathcal{O}_k=\mathbb{Z}[y]/r(y)$. For simplicity we assume also that $h(k)=1$, since this is the case for every Salem number we are interested in.
For $h(k)>1$ the arguments can be adapted.
\begin{enumerate}
\item Construct the principal isometry $f_0: L_0 \rightarrow L_0$ with characteristic polynomial $s(x)$.
\item Compute the set $P$ consisting of the primes in $\O_k$ lying over the feasible primes for $s(x)$ and add to $P$ the primes in $\O_k$ above the characteristic $p$ of the prime field.
\item Let $A$ be the set consisting of those $a\in \O_k$ which are a product of the primes in $P$ and satisfy $D(n)\leq 22-d$, where $n$ is the order of $\overline{f_0}|pA_{L_0(a)}$. This set is finite in virtue of Proposition \ref{prop:finite number of  twists}.
\item Replace $A$ with the subset of those $a \in A$ which satisfy $\mu(\overline{f_0}|pA_{L_0(a)})=0$ for some product $\mu$ of distinct cyclotomic polynomials of degree at most $22-d$.
\item If $p$ is not feasible, keep only those $a \in A$ such that the minimal polynomial of $\overline{f_0}|(A_{L_0(a)})_p$ is irreducible in $\FF_p[x]$.
\item Denote by $U\subseteq \O_k^\times$ a set of representatives of ${\O_k^\times / \O_k^{\times 2}}$ and replace $A$ with the set of those $au\in AU$ such that the signature of $L_0(au)$ is $(1,d-1)$.
\item Replace $A$ with the subset of those $a\in A$ such that $f_0|L_0(a)$ is positive by the quadratic positivity test.
\item Find an $a\in A$, a negative definite lattice $C$ of rank $22-d$ and a positive $f_C \in O(C)$ that can be glued to $(L_0(a),f_0)$ to obtain a positive isometry of a supersingular K3 lattice.
\end{enumerate}
Steps (1)-(7) are easily implemented on a computer algebra system. 
Although step (8) is finite in principle, computations are only feasible for small ranks of $C$. Indeed, at this point we have only a finite number of possibilities for the genus of $C$, and each genus contains only a finite number of classes and each class has only a finite number of isometries. Each of these enumerations can be obtained explicitly (there are implementations for example in Magma), but computation times grow rapidly with the rank of $C$.\\

To illustrate our results we apply the strategy above to determine which minimal Salem numbers $\lambda_d$ are realized in characteristic $5$. The reason to choose $5$ is that it is the smallest for which the Torelli theorems are available. In principle any other $p>3$ is possible.
The constructions are mostly carried out with a package developed by the first author for the computer algebra system SageMath \cite{sage}, while computations for positivity are done with SCIP \cite{scip} and CPLEX \cite{cplex}. The resulting matrices are included in the supplementary data. 

Following McMullen, we represent a gluing between two (or more) lattices by a diagram. Each node consists of an invariant sublattice. If the sublattices corresponding to two nodes are glued, they are joined by an edge which is decorated by the cardinality of the glue. 
  
\begin{theorem}
 The value $\log \lambda_d$ arises as the entropy of an automorphism of a supersingular K3 surface over a field of characteristic $p=5$ if and only if $d\leq 22$ is even and $d\neq 18$. 
\end{theorem}
To prove the theorem we consider each minimal Salem number $\lambda_d$ separately.
\begin{proposition}
 The minimal Salem number $\lambda_{22}$ in degree $22$ is realized on a supersingular K3 surface with Artin invariant $\sigma=4$ and $\sigma=7$ in characteristic $5$. 
\end{proposition}
\begin{proof}
 Since the Salem factor is of degree $22$, no gluing is required.
 The principal $s_{22}$-lattice is unimodular and $5$ factors in $\O_k$ as a product of two primes $p_1p_2$ of norms $5^4$ and $5^7$.
 Both $p_1$ and $p_2$ stay prime in $\mathcal{O}_K$. Indeed, $\overline{s}_{22}$ factors modulo $5$ as a product $\overline{g}_1(x)\overline{g}_2(x)$ of irreducible polynomials $\overline{g}_i(x) \in \mathbb{F}_5[x]$ of degree $8$ and $14$. Therefore $p_i\O_K=(5,g_i(x))$ where the characteristic polynomial of $f_0|L_0(p_i)$ is $\overline{g}_i(x)$. In particular it is irreducible. 
 To conclude, one computes units $u_1,u_2\in \O_k^\times$ such that $(L_0(u_ip_i),f_0)$, $i=1,2$ are hyperbolic, and the linear positivity test confirms the positivity of both constructions. To compute the Artin invariants we use Lemma \ref{lem:dual is fractional ideal} and see that the discriminant group $A_{L_0(p_i)}$ is isomorphic to $O_K/p_i\O_K$. This is indeed a vector space with $|\det L_0 \cdot N^K_\mathbb{Q}(p_i)|=1 \cdot |N^k_\mathbb{Q}(p_i)|^2=5^8$ or $5^{14}$ elements.
 In both cases, Theorem \ref{thm:realized} provides a supersingular K3 surface over $\overline{\mathbb{F}}_5$ (of Artin invariant $4$ respectively $7$) together with an automorphism of entropy $\log \lambda_{22}$. 
\end{proof}

\begin{proposition}
 The minimal Salem number $\lambda_{20}$ of degree $d=20$ is realized in characteristic $5$ with Artin invariant $\sigma=3$ or $\sigma=7$.
\end{proposition}
\begin{proof}
 \FloatBarrier
 \begin{figure}[!htbp]
  \tikzstyle{block} = [draw, rectangle, minimum height=3em, minimum width=3em]
  \tikzstyle{virtual} = [coordinate]

  \begin{tikzpicture}[auto, node distance=2cm]
    \node [block, align=center] (S)     {$s_{20}(x)$-lattice\\
    $\FF_{11}\oplus \FF_5^6$};
    \node [virtual, right of=S](up){};
    \node [virtual, right of=up](up1){};
    \node [block, right of=up1]   (C)     {$(x^2-1)$-lattice $(C,f_C)$};
    \draw [-] (S) -- node {$11$} (C);
  \end{tikzpicture}
  \caption{Gluing for $\lambda_{20}$}\label{fig:glue20}
 \end{figure}
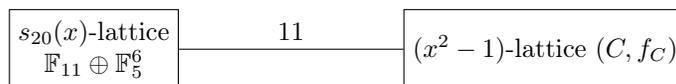 
We construct the isometry $f|\NS$ following the steps in the general strategy above and gluing together two lattices: the Salem factor $S$ and the cyclotomic factor $C$.\\

\textbf{The Salem factor.}
Note that $s_{20}$ is simple, hence $S=L_0(a)$ must be a twist of the principal $s_{20}$ lattice $L_0$, which has determinant $|\det L_0|=|s_{20}(1)s_{20}(-1)|=|-1 \cdot 11|$.
In particular, since modulo $11$
\[\overline{(x+1)} \mid \overline{s_{20}},\]
we see that $11$ is feasible, and in fact it is the only feasible prime.
Therefore the possible twists $a \in A$ must be a product of factors of 11 and $p=5$ in $\mathcal{O}_k$.
In $\O_k$ we have the factorizations $11=a_1a_2$ into two primes of norm $11$ and $11^9$, as well as $5=p_1p_2$ with norms $5^3$ and $5^7$.
On the one hand, a direct computation shows that $|A_{L_0(a_1)}|=11^3$ and $f_0|5A_{L_0(a_1)}$ is of order $22$. 
Since $D(22)=10>2=22-\deg(s_{20})$, $a$ cannot be a multiple of $a_1$, and neither of $a_2$ by the same reasoning.
On the other hand, for any invertible $u \in \O_k^\times /(\O_k^\times)^2$ the quadratic positivity test shows that $f_0|L_0(u)$ is not positive, hence $a$ must be divisible by either $p_1$ or $p_2$. Indeed, for both $p_i$ it is possible to find a unit $u_i$ such that $S=L_0(u_ip_i)$ is hyperbolic and $f_0|L_0(u_ip_i)$ is positive (the linear programming test gives $\mu(f_0|L_0(u_ip_i)) = -4$). Furthermore the $11$-primary part of the discriminant group is $\left(A_{S}\right)_{11}\cong\mathbb{F}_{11}$, the quadratic form is given by $\left(q_S\right)_{11}\left(\overline{x}\right) = \frac{2}{11} \mathbb{Z}$ for a suitable generator $\overline{x}$ and $\left(\overline{f}\right)_{11}$ acts as $-id$.\\

\textbf{The cyclotomic factor.} 
 Since $5$ is not feasible and $\det S$ is divisible by $5$ Corollary \ref{cor:p_not_feasible} implies that $\det C$ is not divisible by $5$.  
 This determines the cyclotomic factor $C$ to be the (unique) negative definite lattice of rank $2$ and determinant $11$. Its Gram matrix and a positive isometry acting as $-id$ on the discriminant group $A_C \cong \mathbb{F}_{11}$ are given by 
 
 \[(C,f_C)=\left[ -\left(
 \begin{matrix}
 2&1\\
 1&6
 \end{matrix}\right), 
  \left(\begin{matrix}
 1 &1\\

 0 &-1\\
 \end{matrix}\right)\right].\]
 For the discriminant form we have $q_C\cong(-2x^2/11)$, hence there is an isomorphism $\phi_{11}: \left(A_S\right)_{11} \rightarrow \left(A_C\right)_{11}$ such that $q_C \circ \phi_{11} = -\left(q_S\right)_{11}$. Hence the gluing of $(S,f_0)$ and $(C,f_C)$ along $\phi_{11}$ results in a lattice $(N,f)$ of signature $(1,21)$ and discriminant $5^6$ (or $5^{14}$). This is represented in Figure \ref{fig:glue20}, where each box represents a sublattice together with its discriminant group and the $\mathbb{F}_{11}$ over the edge represents the glue subgroup. The characteristic polynomial of $\overline{f}$ on $A_N$ is the prime factor of
 $\overline{s_{20}}\in \mathbb{F}_5[x]$ corresponding to the prime $p_1$ (resp. $p_2$), and in particular it is irreducible. Positivity is then verified by the linear programming test. 
 
 In both cases Theorem \ref{thm:realized} provides a supersingular K3 surface over $\overline{\mathbb{F}}_5$ and the automorphism on it.

\end{proof}
  \FloatBarrier

\begin{proposition}
 The minimal Salem number $\lambda_{16}$ of degree $d=16$ is realized in characteristic $5$ with Artin invariant $\sigma=5$.
\end{proposition}
\begin{proof}

 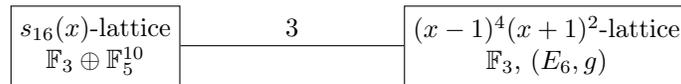
\begin{figure}[!htbp]
  \tikzstyle{block} = [draw, rectangle, minimum height=3em, minimum width=3em]
  \tikzstyle{virtual} = [coordinate]

  \begin{tikzpicture}[auto, node distance=2cm]
    \node [block, align=center] (S)     { $s_{16}(x)$-lattice \\
    $\FF_{3}\oplus \FF_5^{10}$};
    \node [virtual, right of=S](up){};
    \node [virtual, right of=up](up1){};
    \node [block, align=center, right of=up1]   (C)     {$(x-1)^4(x+1)^2$-lattice\\
    $\FF_3$, $(E_6,g)$};
    \draw [-] (S) -- node {$3$} (C);
  \end{tikzpicture}
  \caption{Gluing for $\lambda_{16}$}\label{fig:glue16}
 \end{figure} 
 The feasible primes for $s_{16}$ are $3$ and $29$. 
 At step (7) of the general strategy we are left with twists of norms $3\cdot 5^{5},5^{5},29$. 
 We choose the twist $a$ of norm $5^5$, so that $S=L_0(a)$ has determinant $-3\cdot5^{10}$. In order to remove the $3$-primary part of $A_S$ by gluing, $C$ must have determinant $3$ and signature $(0,6)$, which determines it uniquely as $E_6$ (cf. \cite{conway_sloane:sphere_packings}). A direct computation shows that the forms $(q_S)_3\cong -q_{E_6}$ are opposite and thus a gluing of lattices $N = S \oplus_{\phi_3} C$ exists.
 Since the action of $f_0$ on $(A_S)_3$ is given by $-id$, we need a positive isometry of $E_6$ acting as $-id$ on the discriminant. Looking at the Dynkin diagram of $E_6$, we consider the reflection $h\in O(E_6)$ around the center. A computation shows that $h$ has the desired properties, hence $f_0 \oplus h$ extends to an isometry of $N$ whose positivity is verified by the linear test (with $\mu(f)=-6$).
 The irreducibility of the minimal polynomial on $A_N$ is assured by step $(5)$ of the general strategy and we can apply Theorem \ref{thm:realized} to conclude the proof.
 \end{proof}
 Here is why we choose the twist of norm $5^5$ for the Salem factor:
 If instead we twist the Salem factor above $29$, the only possibility for the cyclotomic part is $c(x)=c_7(x)$. It is a simple reciprocal polynomial. Hence $C$ is a twist of the principal $c_7$-lattice. But $c_7(1)=7$, so it is ramified over $7$ and $7 \mid \det C$. This leads to a contradiction since $7$ is not feasible.
 Since the principal $s_{16}$-lattice $L_0$ is ramified over $3$ (has determinant $\pm 3$), it is simpler to twist just above $5^5$ than $3 \cdot 5^5$.  
\FloatBarrier

\begin{proposition}
 The Salem number $\lambda_{14}$ is realized on a supersingular K3 surface in characteristic $5$ with Artin invariant $\sigma=6$.
\end{proposition}
\begin{proof}
 The principal $s_{14}(x)$-lattice is unimodular. Now we can twist it by a prime $b\in \mathcal{O}_k$ of norm $5^6$ inert in $\mathcal{O}_K$ to get a positive isometry on a $5-$elementary hyperbolic lattice of rank $14$. Since the prime is inert, the characteristic polynomial on the discriminant is irreducible.
 To obtain a hyperbolic lattice of rank $22$ we take the direct sum with $(E_8,id)$, obtaining also a positive isometry. As usual Theorem \ref{thm:realized} provides the supersingular K3 surface and the automorphism. 
 \end{proof}

\begin{proposition}
 The Salem number $\lambda_{12}$ is realized on a supersingular K3 surface with Artin invariant $\sigma=2$ in characteristic $5$. 
\end{proposition}
\begin{proof}
 \begin{figure}[!htbp]
  \tikzstyle{block} = [draw, rectangle, minimum height=3em, minimum width=3em]
  \tikzstyle{virtual} = [coordinate]

  \begin{tikzpicture}[auto, node distance=2cm]
    \node [block,align=center] (C30)    {$c_{30}(x)$-lattice\\ $\mathbb{F}_{31}^2$};
    \node [virtual, right of=C30](up0){};
    \node [block, right of=up0,align=center](S){$s_{12}(x)$-lattice\\ 
    $\FF_5^4 \oplus \FF_7 \oplus \FF_{31}^2$};
    \node [virtual, right of=S](up){};
    \node [block,align=center, right of=up]   (C)     {$(x^2-1)$-lattice\\ $(M,f_M)$, $\FF_7$ };
    \draw [-] (S) -- node {$7$}(C);
    \draw [-] (C30) -- node {$31^2$}(S);
  \end{tikzpicture}
  \caption{Gluing for $\lambda_{12}$}\label{fig:glue12}
 \end{figure}
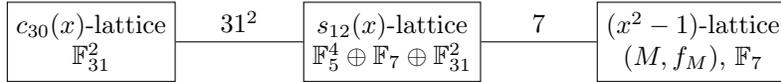

The principal $s_{12}(x)$-lattice $L_0$ has determinant $|s_{12}(1)s_{12}(-1)|=7$, hence discriminant group $\mathbb{F}_7$, where the isometry acts as $-id$. The feasible primes are $7,13,31$. 
Note that $s_{12}$ and $c_{30}$ have the common factor $(x+7)(x+9)$ modulo $31$. 
Hence we choose to twist the principal $s_{12}(x)$-lattice with a prime $q_1 \in \mathcal{O}_k$ of norm $5^2$ inert in $\mathcal{O}_K$ and a prime $q_2$ of norm $31$ such
that $S=L_0(uq_1q_2)$ has signature $(1,11)$ for a suitable unit $u$. Then the discriminant group $A_{S}=\mathbb{F}_{5}^4 \oplus \mathbb{F}_7\oplus \FF_{31}^2$. 
In order to glue over the $7$-primary summand, note that the discriminant form on $\left(A_S\right)_7$ is a square. Hence $S$ can be glued with the negative definite lattice
 \[ (M,f_M)=\left[ -\left(\begin{matrix}
      2&1\\
      1&4
     \end{matrix}\right),
\left(\begin{matrix}
      1&1\\
      0&-1
     \end{matrix}\right) \right] \] because $A_{M} \cong \mathbb{F}_7$ and $-q_M$ is a square. Call the resulting lattice $(L_1,f_1)$.
For the glue above $31$ we take a twist of the principal $c_{30}$-lattice $(C30,f_{C30})$. Now \cite[Theorem 4.3]{mcmullen:entropy_and_glue} guarantees the existence of a twist by a divisor $a$ of $31$  such that the characteristic polynomial of $\overline{f_{C30}}$ on $A_{C30(a)}$ is $(x+7)(x+9)$. We can find a unit $u$ such that $C30(ua)$ has signature $(0,8)$.
By construction the characteristic polynomials on the $31$-primary part match, and \cite[Theorem 3.1]{mcmullen:entropy_and_glue} provides the existence of a glue map $\phi: (A_{L_{1}})_{31} \rightarrow (A_{C30(ua)})_{31}$. Set $(N,f)=(L_1\oplus_\phi C30(ua),f_1 \oplus f_{C30})$, which is a hyperbolic $5$-elementary lattice of determinant $-5^4$.
The linear positivity test of $(N,f)$ fails, since there is a optimal state with objective $-2$, but the quadratic test does confirm the positivity of $(N,f)$. 
To apply the crystalline Torelli theorem it suffices to check that the characteristic polynomial on $N^\vee/N\cong\mathbb{F}_{5}^4$ is irreducible. This is indeed the case, since the twist $q_1$ remains inert in $\mathcal{O}_K$. 
\end{proof}

\begin{proposition}
 Lehmer's number $\lambda_{10}$ is realized by an automorphism of a supersingular K3 surface in characteristic $5$ with Artin invariant $\sigma=2$.

\end{proposition}

\begin{proof}
 \begin{figure}[!htbp]
  \tikzstyle{block} = [draw, rectangle, minimum height=3em, minimum width=3em]
  \tikzstyle{virtual} = [coordinate]

  \begin{tikzpicture}[auto, node distance=2.3cm]
    \node [block, align=center]                 (S)     {$s_{10}(x)$-lattice\\ $\FF_{13}^2\oplus \FF_5^4$};
    \node [virtual, right of=S](up){};
    \node [block, right of=up, align=center]   (C)     {$c_{14}(x)$-lattice\\ $\mathbb{F}_{13}^2\oplus \mathbb{F}_{7}$};
    \node [virtual, right of=C] (C1) {};
    \node [block, right of=C1, align=center] (C1C2)  {$(x^2-1)^3$-lattice\\ 
    $(A_6,g)$, $\mathbb{F}_{7}$};
     Connect nodes
    \draw [-] (S) -- node {$13^2$} (C);
    \draw [-] (C) -- node {$7$} (C1C2);
  \end{tikzpicture}
  \caption{Gluings for $\lambda_{10}$.}\label{fig:glue10}
 \end{figure}
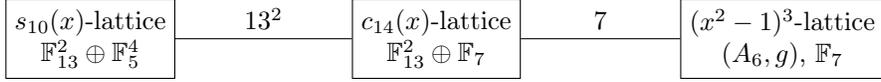 

The principal $s_{10}$-lattice is unimodular and the feasible primes for $s_{10}$ are $3,5,13,23,29$. There is an element $a \in \O_k$ of norm
$5^2 \cdot 13$ such that $S=L_0(a)$ is hyperbolic and $A_S \cong \mathbb{F}_5^4 \oplus \mathbb{F}_{13}^2$ We need to glue $S$ with two negative definite lattices of rank 6 to cancel the $13$-primary part of the discriminant group, as follows.
The only possibility to glue above $13$ is to use the principal $c_{14}$-lattice $C14$, which has discriminant $\mathbb{F}_7$.
Since $c_{14}$ is also simple (with the analogous definition for cyclotomic polynomials) we can find a negative definite twist $C14(b)$ with determinant $7 \cdot 13^2$, and such that the characteristic polynomial of $\overline{f_{14}}$ on the $13$-primary part matches with that of $\overline{f_0}$.
Call $\left(L_1,f_1\right)$ the resulting glue $S \oplus_{\phi_{13}} C14(b)$ over $13$. It has rank $16$ and determinant $5^4 7$, hence it remains to glue it with a negative
definite lattice of rank $6$ and determinant $7$, i.e. $A_6$.
It also remains to find a good isometry $g$ of $A_6$. Since $(A_S)_7\cong\FF_7$ and $\overline{f_1}$ acts as $-id$, so should do $\overline{g}$. The obvious choice $g=-id_{A_6}$ glues just fine, however, it is not positive, as any root of $A_6$ is cyclic, hence we need to look for another one. Let $r_1,\ldots,r_6$ be a set of linearly independent roots (corresponding to the nodes of the Dynkin diagram of $A_6$), then $g$ is given by the central reflection composed with $-id$
\[g: (r_1,\ldots,r_6) \mapsto (-r_6,\ldots,-r_1).\]
A direct computation shows that $g$ has the right properties. 
Since $a$ is inert in $\O_K$, the resulting isometry has irreducible characteristic polynomial and the proof concludes as the preceding ones.
\end{proof}

\begin{proposition}
There is a supersingular K3 surface over $\overline{\FF}_5$ with Artin invariant $\sigma=4$ and an automorphism on it realizing $\lambda_8$. The characteristic polynomial of the action on $\NS$ is given by $s_{8}c_1^{12}c_2^2$.
\end{proposition}
\begin{proof}

 \begin{figure}[!htbp]
  \tikzstyle{block} = [draw, rectangle, minimum height=3em, minimum width=3em]
  \tikzstyle{virtual} = [coordinate]

  \begin{tikzpicture}[auto, node distance=2.3cm]
    \node [block, align=center]                 (S)     {$s_8(x)$-lattice\\ $\FF_{3}\oplus \FF_5^8$};
    \node [virtual, right of=S](up){};
    \node [block, right of=up, align=center]   (C)     {$c_2(x)$-lattice\\ $(E_6,h)$, $\mathbb{F}_{3}$};
    \node [virtual, right of=C] (C1) {};
    \node [block, right of=C1, align=center] (C1C2)  {$(x-1)^8$-lattice\\ 
    $(E_8,id)$};
     Connect nodes
    \draw [-] (S) -- node {$3$} (C);
  \end{tikzpicture}
  \caption{Gluings for $\lambda_{8}$.}\label{fig:glue8}
 \end{figure}
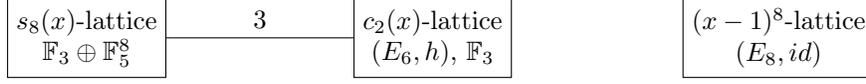 

The principal $s_8(x)$-lattice $L_0$ has discriminant group $\mathbb{F}_3$ and $p=5$ stays prime in $\O_K$.
One can find a unit $u \in \O_K$ such that $S=L_0(5u)$ is hyperbolic.
Then $\det S=-3\cdot 5^8$ and $f_0$ acts as $-id$ on $(A_S)_3\cong \FF_3$.
It turns out we can glue this to $(E_6,h)$ where $h\in O(E_6)$ is given by the central symmetry of the Dynkin diagram of $E_6$ like for $\lambda_{16}$.
Now $(S\oplus_\phi E_6,f_0 \oplus h)$ is a lattice of signature $(1,13)$ and discriminant group $\FF_5^8$. Since $5$ is prime in $\O_K$, $\overline{s}_8$ is the irreducible characteristic polynomial of the action on the discriminant group.
Positivity is confirmed by the linear test, and we conclude by taking the direct sum with 
$(E_8,id)$ to obtain an hyperbolic $5$-elementary lattice of rank $22$. 
\end{proof}

\begin{proposition}
There is a supersingular K3 surface over $\overline{\FF}_5$ with Artin invariant $\sigma=4$ and an automorphism on it realizing $\lambda_6$. Its characteristic polynomial on $\NS$ is given by $s_{6}(x)c_1^{9}(x)c_2(x)c_{14}(x)$.
\end{proposition}
\begin{proof}
 
  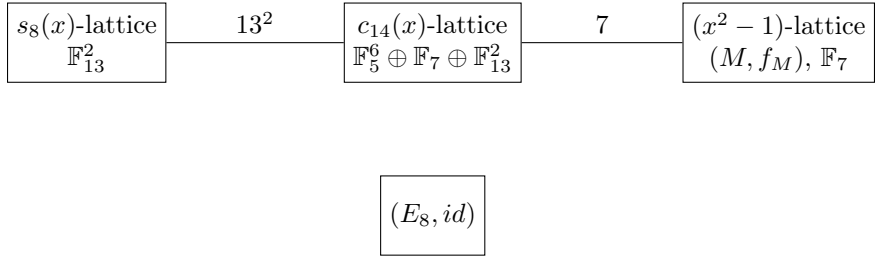
\begin{figure}[!htbp]
   \tikzstyle{block} = [draw, rectangle, minimum height=3em, minimum width=3em]
   \tikzstyle{virtual} = [coordinate]
 
   \begin{tikzpicture}[auto, node distance=2.3cm]
     \node [block, align=center]                 (S)     {$s_{8}(x)$-lattice\\ $\FF_{13}^2$};
     \node [virtual, right of=S](up){};
     \node [block, right of=up, align=center]   (C)     {$c_{14}(x)$-lattice \\ $\mathbb{F}_5^6 \oplus \mathbb{F}_{7} \oplus \mathbb{F}_{13}^2$};
     \node [virtual, right of=C] (C1) {};
     \node [block, right of=C1, align=center] (C1C2)  {$(x^2-1)$-lattice\\ 
     $(M,f_M)$, $\mathbb{F}_{7}$};
     \node [block, below of=C,align=center] (E8) {$(E_8,id)$};
      Connect nodes
     \draw [-] (S) -- node {$13^2$} (C);
     \draw [-] (C) -- node {$7$} (C1C2);
   \end{tikzpicture}
   \caption{Gluings for $\lambda_{6}$.}\label{fig:glue6}
  \end{figure} 

 The principal $s_6(x)$-lattice $L_0$ is unimodular, and the feasible primes are $2, 3, 7, 13, 23, 29, 31, 37, 41, 59, 67$.
 We choose to twist $L_0$ by a prime $q \in \O_K$ of norm $13$
 such that $S=L_0(q)$ is hyperbolic and $\overline{f}|A_S$ has characteristic polynomial $$x^2+8x+1=\gcd(\overline{s}_6,\overline{c}_{14}) \mod 13.$$ This suggests to glue $S$ with a twist of the principal $c_{14}$-lattice $(C14,f_{14})$. By \cite[Theorem 4.3]{mcmullen:entropy_and_glue} we can find a twist $b \in \ZZ[\zeta_{14}]$ dividing $13$ with the right characteristic polynomial on the $13$-primary part of the discriminant. We can even arrange for $C14(b)$ to be negative definite. Since $5$ is prime in $\ZZ[\zeta_{14}]$, we can take the further twist $C14(5b)$ to get $\left(A_{C14(5b)}\right)_5 \cong \mathbb{F}_5^6$ and $\left(\overline{f_{14}}\right)_5$ with irreducible characteristic polynomial. 
 Now \cite[Theorem 3.1]{mcmullen:entropy_and_glue} provides the existence of a glue map 
 $\phi: A_S \rightarrow (A_{C14(5b)})_{13}$ compatible with the actions. 
 Set $N=S \oplus_\psi C14(5b)$. It is a hyperbolic lattice of rank $12$ and determinant $-5^{6}7$ with order $2$ action on $(A_N)_7$. Then $(N,f_N)$ turn out to glue to
 \[(M,f_M)=\left[-\left(\begin{matrix}
   2 & 1\\
   1 & 4
  \end{matrix}\right),\left(\begin{matrix}
    1 & 1\\
    0 &-1
   \end{matrix}\right)\right] \]
  We conclude by confirming positivity and filling up the remaining rank $8$ by $(E_8,id)$. 
 \end{proof}

\begin{proposition}
 The Salem numbers $\lambda_{2}$ and $\lambda_4$ are realized in the supersingular K3 surface of Artin invariant $\sigma=1$ in characteristic 5.
\end{proposition}
\begin{proof}
For these Salem numbers of small degree we follow a different strategy, along the lines of the proof of \cite[Theorem 1.3]{mcmullen:entropy_and_glue}, which gives a more explicit construction of the automorphisms not relying on the Torelli theorem.

First of all, note that the supersingular K3 surface with Artin invariant $\sigma=1$ over $\overline{\mathbb{F}_5}$ is the Kummer surface associated to the product of any two supersingular elliptic curves. For example we can consider the reduction modulo 5 of $E=E_{\zeta_3}$, the complex elliptic curve of $j$-invariant 0. By general theory, if $X$ is any smooth projective variety with an automorphism $F$ defined over $\mathbb{Q}$, the entropy of $F|X(\mathbb{C})$ coincides with the entropy of $F|X_p$ for any prime $p$ of good reduction (this follows from the standard comparison theorems between singular and \'etale cohomologies and the properties of good reduction).

Therefore, it is enough to construct automorphisms of ${\rm Km}(E\times E)$ with entropies $\lambda_2$ and $\lambda_4$. Moreover, according to the discussion in \cite[Section 4]{mcmullen:siegel_disk}, it is enough to construct linear maps $F_2, F_4 : \mathbb{C}^2 \rightarrow \mathbb{C}^2$ preserving the lattices $\mathbb{Z}\left[\zeta_3\right]^2$ whose spectral radii $\rho_2,\rho_4$ satisfy $\left|\rho_i\right|^2=\lambda_i$. This is achieved, for example, by the matrices
$$F_4 = \left(\begin{array}{cc} 1 & \zeta_3 - 1\\ -1 & 0 \end{array}\right) \quad \text{and} \quad F_2 = \left(\begin{array}{cc} 1 & 1 \\ 1 & 0 \end{array}\right).$$
\end{proof}

\begin{proposition}
 The supersingular K3 surface $X$ with Artin invariant $\sigma=1$ over $\overline{\mathbb{F}}_{11}$ 
 admits an automorphism $F:X \rightarrow X$ such that the characteristic polynomial of $F^*|\NS(X)$ is given 
 by $s_{18}(x)c_{12}(x)$. 
 It is not realized on a supersingular K3 surface in characteristic $5$. 
\end{proposition}
\begin{proof}
  We begin by proving that $\lambda_{18}$ is not realized in characteristic $p=5$. 
  The principal $s_{18}$-lattice is unimodular, and the feasible primes are $7$ and $13$. 
  By the time we reach step (7) of the general strategy we are left with a single twist $a$ (up to units) of norm $13$.
  Then the only possibility for the cyclotomic factor $c(x)$ is $c_{12}=x^4-x^2+1$, which is a simple reciprocal polynomial.
  Hence $C$ must be a twist of the principal $c_{12}$-lattice by factors of $5$ and $13$.
  However, $5$ is prime in the trace field $k = \mathbb{Q}[y]/r_{18}(y)$, but splits in the Salem field $K=\mathbb{Q}[x]/s_{18}(x)$. 
  This results in the minimal polynomial on the $5$-discriminant group being reducible. 
  In consequence $\lambda_{18}$ is not realizable in characteristic $5$.
  
  However in characteristic $11$ this is possible. We can find a twist $b$ of the principal $c_{12}$-lattice $C12$ 
  such that $C12(b)$ is negative definite, $A_{C12(b)}=\FF_{11}^2\oplus \FF_{13}^2$, the characteristic polynomials
  on $\FF_{13}^2$ match  and the characteristic polynomial of $\overline{f_1}$ on $\FF_{11}^2$ is irreducible.
  We get the existence of a gluing $N=S\oplus_\psi C12(b)$ along $13^2$ such that $A_N=(A_{C12(b)})_{11}$.
  Positivity of the resulting isometry is confirmed by the (quadratic) test. 
\end{proof}

 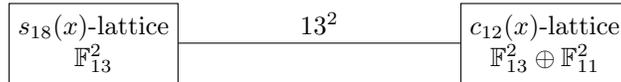
\begin{figure}[!htbp]
  \tikzstyle{block} = [draw, rectangle, minimum height=3em, minimum width=3em]
  \tikzstyle{virtual} = [coordinate]

  \begin{tikzpicture}[auto, node distance=2cm]
    \node [block, align=center]                 (S)     {$s_{18}(x)$-lattice\\ $\mathbb{F}_{13}^2$};
    \node [virtual, right of=S](up){};
    \node [virtual, right of=up](up1){};
    \node [block, right of=up1, align=center]   (C)     {$c_{12}(x)$-lattice\\ $\mathbb{F}_{13}^2\oplus \mathbb{F}_{11}^2$};
    \draw [-] (S) -- node {$13^2$} (C);
  \end{tikzpicture}
  \caption{Gluing for $\lambda_{18}$.}\label{fig:glue18}
 \end{figure} 

\appendix

\section{Realizing $\lambda_{12}$ over $\mathbb{C}$} \label{app:lambda12_complex}
 \begin{theorem}
  There is a complex projective K3-surface $X$ and $F \in Aut(X)$ such that $h(F)=\log \lambda_{12}$, $\NS(X)\cong U(13) \oplus 2 E_8$ and the action on the holomorphic $2$-form is of order $12$. 
 \end{theorem}
 \begin{proof}
 For $s_{12}$ we get the $3$ feasible primes $7,13,31$. Following the general strategy in the complex case, we end up with three twists (up to units) one above each feasible prime. We continue with the twist above $13$, as $7$ and $31$ lead to many dead ends. 
 Modulo $13$ we find the common factor $\overline{(x+2)(x+7)}$ of $s_{12}$ and $c_{12}$.
 By \cite[Theorem 4.3]{mcmullen:entropy_and_glue}º we can find twists $a,b$ above $13$ of the principal $s_{12}$-lattice $L_0$, and the principal $c_{12}$-lattice $C12$ such that they have characteristic polynomial $\overline{(x+2)(x+7)}$ on the $13$-glue.
 Then \cite[Theorem 3.1]{mcmullen:entropy_and_glue} provides the existence of a glue map. 
 It remains to modify $a$ and $b$ by a unit to obtain the right signatures. 
 Indeed for $a$ one can find a unit $u\in \mathcal{O}_k^\times$ such that $S=L_0(ua)$ is of signature $(1,11)$. 
 For $c_{12}$ it is not possible to realize glue group $\FF_{13}^2$ and signature $(0,4)$ but it is possible to achieve 
 signature $(2,2)$. This indicates that we should take $C12(b)$ as transcendental lattice.
 Since $\det S=|\det L_0 N(a)|=7\cdot 13^2$, the only possibility for the remaining part is a negative definite rank $6$ lattice of determinant $7$, i.e. the $A_6$ root lattice. And indeed the quadratic forms $(q_S)_7\cong -(q_{A_6})$ glue.
 Since the characteristic polynomial of $f|T(X)$ is a perfect power, it must be a part of $\NS$. What remains is to find a good positive isometry $g$ of $A_6$. Since $(A_S)_7\cong\FF_7$ $\overline{f}$ acts as $-id$, so does $g$ and we can take the pair $(A_6,g)$ from the construction of Lehmer's number.
  
   \begin{figure}[!htbp]
  \tikzstyle{block} = [draw, rectangle, minimum height=3em, minimum width=3em]
  \tikzstyle{virtual} = [coordinate]

  \begin{tikzpicture}[auto, node distance=2cm]
    \node [block, align=center]                 (S)     {$s_{12}(x)$-lattice\\ $(1,11)$, $\mathbb{F}_{13}^2\oplus \FF_7$};
    \node [virtual, right of=S](up){};
    \node [virtual, left of=S](up1){};
    \node [block, left of=up1, align=center]   (C12)     {$c_{12}(x)$-lattice \\ $(2,2)$,  $\mathbb{F}_{13}^2$};
    \node [block, right of=up, align=center]   (A6)     {$(x^2-1)^3$-lattice\\ $(A_6,g)$, $(0,6)$ $\mathbb{F}_{7}$};
    \draw [-] (C12) -- node {$13^2$} (S);
    \draw [-] (S) -- node {$7$} (A6);
  \end{tikzpicture}
  \caption{Gluing for $\lambda_{12}$ in the complex case.}\label{fig:glue12complex}
 \end{figure}
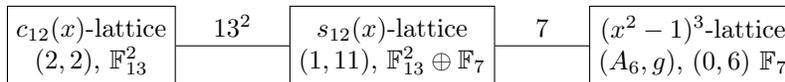 
 The positvity of the isometry on $\NS$ is confirmed by the quadratic positivity test. 
 Note that by Corollary \ref{coro:positivity_finite_part_not_important} we could have taken any other positive $g\in O(A_6)$, acting as $-id$ on the discriminant. The lattice $A_6$ has only $10080$ isometries so a brute-force search is feasible and returns about a hundred suitable isometries.
\end{proof}

\FloatBarrier
\bibliographystyle{JHEPsort}
\bibliography{literature.bib}{}

\end{document}